\documentclass[11pt]{article}
\usepackage{amsmath}
\usepackage{amssymb}
\usepackage{amsthm}
\usepackage[usenames]{color}
\usepackage{amscd}
\usepackage{dsfont}
\usepackage{indentfirst}

\usepackage[colorlinks=true,linkcolor=blue,filecolor=red,
citecolor=webgreen]{hyperref}
\definecolor{webgreen}{rgb}{0,.5,0}
\definecolor{amber}{rgb}{1.0, 0.49, 0.0}

\def\C{{\mathds{C}}}
\def\Q{{\mathds{Q}}}
\def\R{{\mathds{R}}}
\def\N{{\mathds{N}}}
\def\Z{{\mathds{Z}}}
\def\P{{\mathds{P}}}
\def\F{{\mathds{F}}}

\def\fn{{\mathfrak{n}}}
\def\fd{{\mathfrak{d}}}
\def\fD{{\mathfrak{D}}}

\def\1{{\bf 1}}

\def\Cik{\operatorname{C}}
\def\U{\operatorname{U}}

\def\id{\operatorname{id}}
\def\lcm{\operatorname{lcm}}
\def\Reg{\operatorname{Reg}}
\def\GL{\operatorname{GL}}

\def\pont{$\bullet$ }

\newtheorem{theorem}{Theorem}[section]

\newtheorem{lemma}[theorem]{Lemma}
\newtheorem{corollary}[theorem]{Corollary}

\begin{document}

\title{{\bf Proofs, generalizations and analogs of Menon's identity: a survey}}
\author{L\'aszl\'o T\'oth\thanks{The research was financed by NKFIH in Hungary, within the framework of 
the 2020-4.1.1-TKP2020 3rd thematic programme of the University of P\'ecs.}
\\ Department of Mathematics \\
University of P\'ecs \\
Ifj\'us\'ag \'utja 6, 7624 P\'ecs, Hungary \\
E-mail: {\tt ltoth@gamma.ttk.pte.hu} }
\date{}
\maketitle

\centerline{\sl Acta Univ. Sapientiae, Mathematica, {\bf 15} (2023), no. 1, 142--197}
\vskip2mm

\centerline{\sf Dedicated to my father, L\'aszl\'o~Gy.~T\'oth, high school mathematics teacher,}
\centerline{\sf on the occasion of his 90th birthday}

\begin{abstract} Menon's identity states that for every positive integer $n$ one has 
$\sum (a-1,n) = \varphi(n) \tau(n)$, where $a$ runs through a reduced residue system (mod $n$), 
$(a-1,n)$ stands for the greatest common divisor of $a-1$ and $n$,
$\varphi(n)$ is Euler's totient function and $\tau(n)$ is the number of divisors of $n$. 
Menon's identity has been the subject of many research papers, also in the last years.
We present detailed, self contained proofs of this identity
by using different methods, and point out those that we could not identify in the literature.
We survey the generalizations and analogs, and overview the results 
and proofs given by Menon in his original paper. Some historical remarks and  
an updated list of references are included as well.
\end{abstract}



\tableofcontents

\newpage

\section{Introduction}

\subsection{Motivation}

Menon's identity states that for every $n\in \N:=\{1,2,\ldots \}$,
\begin{equation} \label{Menon_id}
M(n):= \sum_{\substack{a\, \text{(mod $n$)} \\ (a,n)=1}} (a-1,n) = \varphi(n) \tau(n),
\end{equation} 
where $a$ runs through a reduced residue system (mod $n$), $(k,n)$ stands for the greatest common divisor of $k$ and $n$,
$\varphi(n)$ is Euler's totient function, and $\tau(n)=\sum_{d\mid n} 1$ is the divisor function. Identity \eqref{Menon_id}
was proved by the Indian mathematician P. K. Menon\footnote{Puliyakot Kesava Menon
(1917-1979)} \cite{Men1965} in 1965. Note that
\eqref{Menon_id} is implicitly contained in a more general earlier result, regarding $n$-even functions, due to 
E. Cohen\footnote{Eckford Cohen (1920-2005), an American mathematician}
\cite{Coh1959II} in 1959. See Section \ref{Section_n_even}.
Other notations, used in the literature, for the above sum are $\sum_{a\in \U(\Z_n)}$ and $\sum_{a\in \Z_n^*}$, where  $\U(\Z_n)= \Z_n^*$ denotes the 
group of units of the ring $\Z_n=\Z/n\Z$ of residue classes (mod $n$). However, we prefer to use the notation in \eqref{Menon_id}.

This arithmetic identity is interesting in itself, it is related to certain num\-ber-theoretic, group-theoretic and combina\-to\-ri\-al properties, and 
has several different proofs by using pure number-theoretic arguments, 
the orbit counting lemma, the finite Fourier representation of $n$-even functions and Vai\-dy\-ana\-thas\-wa\-my's\footnote{Ramaswamy S. Vaidyanathaswamy (1894-1960), an 
Indian mathematician} class division of integers (mod $n$), respectively. Also, there are many possibilities to obtain generalizations and analogs,
which can also be combined. For example: How to modify identity \eqref{Menon_id} if $(a-1,n)$ is replaced by $f((a-1,n))$, where $f$ is an arithmetic function, or by $(P(a),n)$, 
where $P$ is an arbitrary polynomial with integer coefficients, in particular $P(a)=a-s$, with $s$ a given integer? As answers to these questions see identities \eqref{Menon_f_convo_form}, 
\eqref{id_by_Vaidya_method}, \eqref{Richards_id}, \eqref{Menon_id_s} and \eqref{Richards_id_f}.

The sum $M(n)$ is related to the gcd-sum function. See Section \ref{Section_Gcd_sum_function}.
Another identity, similar to \eqref{Menon_id}, states that for every $n\in \N$ and $s\in \Z$,
\begin{equation} \label{Ramanujan_sum_id}
R_s(n):= \sum_{\substack{a\, \text{(mod $n$)} \\ (a,n)=1}} c_n(a-s) = \mu(n) c_n(s),
\end{equation}
obtained in 1959 by Cohen \cite[Cor.\ 7.2]{Coh1959II}, where $\mu(n)$ is the M\"{o}bius function and $c_n(s)$ is the
Ramanujan sum\footnote{named after Srinivasa Ramanujan (1987-1920), the most famous Indian mathematician}, 
defined by
\begin{equation}  \label{Ramanujan_sum}
c_n(k)= \sum_{\substack{a\, \text{(mod $n$)} \\ (a,n)=1}} e(ak/n),
\end{equation}
with $e(x)=e^{2\pi ix}$.

For a common generalization of \eqref{Menon_id} and \eqref{Ramanujan_sum_id} in terms of $n$-even functions, in connection with the Brauer-Rademacher identity see \eqref{Menon_even}, \eqref{Menon_id_s} and \eqref{Ramanujan_id_3_terms}.

Menon's identity has been reproved and generalized in various directions by several authors, also in many recent papers.
However, it seems that the results and proofs in Menon's original paper and in some other papers are not widely known.
The goals of the present survey are to present detailed proofs of this identity by using different methods, and to review its generalizations and analogs. 
We point out those proofs and formulas that we could not identify in the literature. We only discuss the directions of generalizations, and present certain identities 
obtained by one or two steps of generalization (some of them with proofs), but not the most general ones. We also overview the results and proofs given by Menon in 
his paper \cite{Men1965}, and offer some further historical remarks. An updated list of references is included.

\subsection{The gcd-sum function} \label{Section_Gcd_sum_function}

The gcd-sum function, also called Pillai function\footnote{named after Subbayya Sivasankaranarayana Pillai (1901-1950), another Indian mathematician} is
\begin{equation} \label{Pillai}
P(n):= \sum_{a\, \text{(mod $n$)}} (a,n)= \sum_{d\mid n} d \varphi(n/d) \quad (n\in \N).
\end{equation}

The given representation follows by grouping the terms according to the values $(a,n)=d$. Alternatively, using the Gauss formula $n=\sum_{d\mid n} \varphi(d)$ we have
\begin{equation*} 
P(n)= \sum_{a=1}^n \sum_{d\mid (a,n)} \varphi(d) = \sum_{a=1}^n \sum_{\substack{d\mid a\\ d\mid n}} \varphi(d) = \sum_{d\mid n} \varphi(d) \sum_{b=1}^{n/d} 1 = n \sum_{d\mid n} \frac{\varphi(d)}{d},
\end{equation*}
showing that $P=\id *\varphi$, in terms of the Dirichlet convolution. It turns out that the function $P$ is multiplicative, 
and direct computations show (see Section \ref{Section_Method_I}) that for every $n\in \N$,
\begin{equation*} 
P(n)= \prod_{p^\nu \mid\mid n} \left((\nu+1)p^{\nu} -\nu p^{\nu-1} \right) =
n \tau(n) \prod_{p^\nu \mid\mid n} \left(1-\frac{\nu/(\nu+1)}{p} \right).
\end{equation*}

See the surveys by Haukkanen \cite{Hau2008} and T\'oth \cite{Tot2010} on further properties of the gcd-sum function and its generalizations and analogs.

\subsection{Notation}

Throughout this survey we use the following main notations. Those not included here are explained in the text.

\pont $\N=\{1,2,\ldots\}$,

\pont $\P$ is the set of primes,

\pont the prime power factorization of $n\in \N$ is $n=\prod_{p\in \P} p^{\nu_p(n)}$, where all but
a finite number of the exponents $\nu_p(n)$ are zero,

\pont $(n_1,\ldots n_k)$ and $\gcd(n_1,\ldots n_k)$ denote the greatest common divisor of $n_1,\ldots,n_k\in \N$,

\pont $[n_1,\ldots,n_k]$ and $\lcm(n_1,\ldots,n_k)$ denote the least common multiple of $n_1,\ldots,n_k\in \N$,

\pont $\Cik_n$ is the cyclic group of order $n$,

\pont $\Z_n=\Z/n\Z$ is the group/ring of residue classes (mod $n$), 

\pont $\U(\Z_n)=\Z_n^*$ is the group of units of the ring $\Z_n$, 

\pont $\1$ is the function $\1(n)=1$ ($n\in \N$),

\pont $\id$ is the function $\id(n)=n$ ($n\in \N$),

\pont $\mu$ is the M\"obius function,

\pont $\omega(n)$ is the number of distinct prime factors of $n\in \N$,

\pont $\sigma_s(n)=\sum_{d\mid n} d^s$ ($s\in \R$),

\pont $\sigma(n)=\sigma_1(n)$ is the sum of divisors of $n$,

\pont $\tau(n)=\sigma_0(n)$ is the number of divisors of $n$,

\pont $J_s$ is the Jordan function of order $s$ given by $J_s(n)=n^s \prod_{p\mid n} (1-1/p^s)$ ($s\in \R$),

\pont $\varphi=J_1$ is Euler's totient function,

\pont $P(n)=\sum_{a=1}^n (a,n)$ is the gcd-sum function,

\pont $e(x)=e^{2\pi ix}$,

\pont $c_q(n)$ are the Ramanujan sums,

\pont $(f*g)(n)=\sum_{d\mid n} f(d)g(n/d)$ is the Dirichlet convolution of the arithmetic functions $f$ and $g$,

\pont $(f \otimes g)(k) = \sum_{a+b\equiv k \text{ (mod $n$)}} f(a)g(b)$ is the Cauchy convolution of the $n$-periodic 
functions $f$ and $g$,

\pont $d\mid\mid n$ means that $d$ is a unitary divisor of $n$, i.e., $d\mid n$ and $(d,n/d)=1$,

\pont $(f\times g)(n)= \sum_{d\mid \mid n} f(d)g(n/d)$ is the unitary convolution of the arithmetic functions $f$ and $g$,

\pont $\sigma^*_s(n)=\sum_{d\mid \mid  n} d^s$ ($s\in \R$),

\pont $\sigma^*(n)=\sigma^*_1(n)$ is the sum of unitary divisors of $n$,

\pont $\tau^*(n)=2^{\omega(n)}$ is the number of unitary divisors of $n$.

\section{Proofs of the classical Menon identity}

In this Section we present direct, self contained proofs of the Menon identity \eqref{Menon_id}. Perhaps the first idea that comes to mind is to show that the given sum is a multiplicative function of $n$ (Method I). The condition $(a,n)=1$ can be treated by the inclusion-exclusion principle (Method II). Alternatively, the M\"obius function, respectively exponential sums can be used, combined with certain convolutional identities (Methods III, IV, V, VII, VIII). Application of the orbit counting lemma provides a different proof (Method VI). Finite Fourier representations of $n$-even functions can also be applied (Method IX). Finally, some  properties related to Vaidyanathaswamy's class division of integers (mod $n$) lead to another proof (Method X).   

A variant of Method I, as well as Methods VI and X have been used by Menon \cite{Men1965}. The approach of Methods IV and V goes back to Sita Ramaiah \cite{Sit1978}. Method IX was first applied by Cohen \cite{Coh1959II} and Nageswara Rao \cite{Nag1972}. See Section \ref{Section_Historical_remarks} for more details. We are not aware of references concerning Methods II, III, VII and VIII. 

\subsection{Method I: Proof by multiplicativity} \label{Section_Method_I}

Both sides of identity \eqref{Menon_id} are multiplicative in $n$. This is clear for the right hand side, and we prove it for the left hand side $M(n)$.
Let $(n_1,n_2)=1$. We apply the following well-known property: If $a_1$ runs through a reduced residue system (mod $n_1$) and $a_2$ runs through a
reduced residue system (mod $n_2$), then $a=a_1n_2+a_2n_1$ runs through a reduced residue system (mod $n_1n_2$). See, e.g.,
Hardy and Wright \cite[Th.\ 61]{HarWri2008}. Hence
\begin{equation*}
M(n_1n_2)= \sum_{\substack{a_1 \text{(mod $n_1$)} \\ (a_1,n_1)=1}} \sum_{\substack{a_2 \text{(mod $n_2$)} \\ (a_2,n_2)=1}} (a_1n_2+a_2n_1-1,n_1n_2)
\end{equation*}
\begin{equation*}
= \sum_{\substack{a_1 \text{(mod $n_1$)} \\ (a_1,n_1)=1}} \sum_{\substack{a_2 \text{(mod $n_2$)}  \\ (a_2,n_2)=1}} (a_1n_2+a_2n_1-1,n_1) (a_1n_2+a_2n_1-1,n_2)
\end{equation*}
\begin{equation*}
= \sum_{\substack{a_1 \text{(mod $n_1$)}  \\ (a_1,n_1)=1}} (a_1n_2-1,n_1) \sum_{\substack{a_2 \text{(mod $n_2$)} \\ (a_2,n_2)=1}} (a_2n_1-1,n_2)=M(n_1)M(n_2),
\end{equation*}
where we use that $a_1n_2$ runs through a reduced residue system (mod $n_1$) and $a_2n_1$ runs through a reduced residue
system (mod $n_2$).

Now, for a prime power $p^\nu$ ($\nu \ge 1$) we have
\begin{equation*}
M(p^\nu)= \sum_{\substack{a\, \text{(mod $p^\nu$)}\\ (a,p)=1}} (a-1,p^\nu) =  \sum_{\substack{a\, \text{(mod $p^\nu$)}}} (a-1,p^\nu)-
\sum_{j=1}^{p^{\nu-1}} (jp-1,p^\nu)
\end{equation*}
\begin{equation*}
=\sum_{a\, \text{(mod $p^\nu$)}} (a,p^\nu)- \sum_{j=1}^{p^{\nu-1}} 1 = P(p^\nu)  - p^{\nu-1}.
\end{equation*}

To compute $P(p^\nu)$, which is the value of the gcd-sum function for $n=p^\nu$, one can use the convolutional representation \eqref{Pillai}. To give here a direct argument, let us group the terms according to the values
$(a,p^\nu)=p^t$ with $0\le t\le \nu$. Then $a=bp^t$, where $1\le b\le p^{\nu-t}$ and $p\nmid b$. The number of such values $b$ is
$\varphi(p^{\nu-t})$. Therefore,
\begin{equation} \label{Pillai_prime_power}
P(p^\nu) =\sum_{t=0}^\nu p^t\varphi(p^{\nu-t})= p^\nu+ \sum_{t=0}^{\nu-1} p^t(p^{\nu-t}-p^{\nu-t-1})= (\nu+1)p^\nu - \nu p^{\nu-1}.
\end{equation}

We deduce that
\begin{equation*}
M(p^\nu) = (\nu+1)p^\nu-\nu p^{\nu-1}- p^{\nu-1}= (\nu+1)(p^\nu-p^{\nu-1})=\varphi(p^\nu)\tau(p^\nu).
\end{equation*}

This completes the proof of \eqref{Menon_id}.

\subsection{Method II: Proof by the inclusion-exclusion principle} \label{Section_Method_II}

Let $n=p_1^{\nu_1}\cdots p_r^{\nu_r} > 1$. We have by the inclusion-exclusion principle, 
\begin{equation*}
M(n)= \sum_{\substack{a=1\\ (a,n)=1}}^n (a-1,n) = \sum_{a=1}^n (a-1,n) - \sum_{1\le j\le r} \sum_{\substack{a=1\\ p_j\mid a }}^n (a-1,n) 
\end{equation*}
\begin{equation*}
+ \sum_{1\le j<k\le r} \sum_{\substack{a=1\\ p_jp_k \mid a}}^n (a-1,n) - \sum_{1\le j<k<\ell \le r} \sum_{\substack{a=1\\ p_jp_k p_{\ell} \mid a}}^n (a-1,n)+ \cdots 
+ (-1)^r \sum_{\substack{a=1\\ p_1\cdots p_r\mid a }}^n (a-1,n).
\end{equation*}

Here
\begin{equation*}
\sum_{a=1}^n (a-1,n) = \sum_{a=1}^n (a,n)=P(n), 
\end{equation*}
the gcd-sum function. Furthermore, for every $j$,
\begin{equation*}
S_j:=\sum_{\substack{a=1\\ p_j\mid a }}^n (a-1,n) = \sum_{b=1}^{n/p_j} (bp_j-1,n/p_j^{\nu_j}),
\end{equation*}
since $(bp_j-1,p_j^{\nu_j})=1$. We deduce
\begin{equation*}
S_j = \sum_{k=1}^{p_j^{\nu_j-1}} \sum_{c=1}^{n/p_j^{\nu_j}} \left( ((k-1)n/p_j^{\nu_j}+c) p_j-1,n/p_j^{\nu_j}\right)
\end{equation*}
\begin{equation*}
= \sum_{k=1}^{p_j^{\nu_j-1}} \sum_{c=1}^{n/p_j^{\nu_j}} \left((k-1)n/p_j^{\nu_j-1}+c p_j -1,n/p_j^{\nu_j}\right) = \sum_{k=1}^{p_j^{\nu_j-1}} \sum_{c=1}^{n/p_j^{\nu_j}} 
\left(c p_j -1,n/p_j^{\nu_j}\right)
\end{equation*}
\begin{equation*}
=  \sum_{k=1}^{p_j^{\nu_j-1}} P(n/p_j^{\nu_j}) = p_j^{\nu_j-1} P(n/p_j^{\nu_j}),
\end{equation*}
using that together with $c$, the values $cp_j$ run through a complete residue system (mod $n/p_j^{\nu_j}$). That is,
\begin{equation*}
\sum_{\substack{a=1\\ p_j\mid a }}^n (a-1,n) = P(n) \frac{p_j^{\nu_j-1}}{P(p_j^{\nu_j})}.
\end{equation*}

In a similar way,
\begin{equation*}
\sum_{\substack{a=1\\ p_jp_k \mid a}}^n (a-1,n) = P(n) \frac{p_j^{\nu_j-1}p_k^{\nu_k-1}}{P(p_j^{\nu_j})P(p_k^{\nu_k})}, 
\end{equation*}
\begin{equation*}
\sum_{\substack{a=1\\ p_jp_kp_{\ell} \mid a}}^n (a-1,n) = P(n) \frac{p_j^{\nu_j-1}p_k^{\nu_k-1} p_{\ell}^{\nu_{\ell}-1}}{P(p_j^{\nu_j})P(p_k^{\nu_k}) P(p_{\ell}^{\nu_{\ell}})},
\end{equation*}
etc., and we have
\begin{equation*}
M(n)= P(n) -  P(n) \sum_{1\le j\le r} \frac{p_j^{\nu_j-1}}{P(p_j^{\nu_j})} 
+ P(n) \sum_{1\le j<k\le r} \frac{p_j^{\nu_j-1}p_k^{\nu_k-1}}{P(p_j^{\nu_j}) P(p_k^{\nu_k})} 
\end{equation*}
\begin{equation*}
- P(n)\sum_{1\le j<k<\ell\le r} \frac{p_j^{\nu_j-1} p_k^{\nu_k-1} p_{\ell}^{\nu_{\ell}-1}} {P(p_j^{\nu_j}) P(p_k^{\nu_k}) P(p_{\ell}^{\nu_{\ell}})}  + \cdots + (-1)^r P(n)\frac{p_1^{\nu_1-1}\cdots p_r^{\nu_r-1} }{P(p_1^{\nu_1})\cdots P(p_r^{\nu_r})}
\end{equation*}
\begin{equation*}
= P(n) \prod_{1\le j\le r} \left(1- \frac{p_j^{\nu_j-1}}{P(p_j^{\nu_j})}\right)
=\prod_{p^\nu \mid\mid n} P(p^\nu) \left(1- \frac{p^{\nu-1}}{P(p^{\nu})}\right)
\end{equation*}
\begin{equation*}
=\prod_{p^\nu \mid\mid n} \left(P(p^\nu) - p^{\nu-1}\right) 
= \prod_{p^\nu \mid\mid n} \left((\nu+1)p^\nu-\nu p^{\nu-1} - p^{\nu-1}\right) =
\end{equation*}
\begin{equation*}
= \prod_{p^\nu \mid\mid n} (\nu+1)(p^\nu- p^{\nu-1}) = \varphi(n)\tau(n),
\end{equation*}
by using \eqref{Pillai_prime_power}. This proof is similar to the well known proof for Euler's $\varphi$ function.

\subsection{Method III: Proof by convolutional identities}

Slightly more generally, let $f$ be an arbitrary arithmetic function. We have, by grouping the terms according to the values
$(a-1,n)=d$,
\begin{equation*}
M_f(n): = \sum_{\substack{a=1\\ (a,n)=1}}^n f((a-1,n)) = \sum_{d\mid n} f(d) \sum_{\substack{a=1\\(a,n)=1\\ d\mid  a-1\\
((a-1)/d,n/d)=1 }}^n 1.
\end{equation*}

Now using the property of the M\"{o}bius $\mu$ function, namely,
\begin{equation*}
\sum_{d\mid n} \mu(d)  =  \begin{cases} 1,  & \text{ if $n=1$},\\
0, & \text{ otherwise}, \end{cases}
\end{equation*}
we have
\begin{equation*}
M_f(n)= \sum_{d\mid n} f(d) \sum_{\substack{a=1\\ (a,n)=1 \\ d\mid  a-1}}^n \sum_{\delta \mid ((a-1)/d,n/d)} \mu(\delta) =
\sum_{d\mid n} f(d) \sum_{\delta\mid  n/d} \mu(\delta) \sum_{\substack{a=1\\ (a,n)=1 \\ d\mid a-1\\ \delta \mid (a-1)/d }}^n 1
\end{equation*}
\begin{equation} \label{sum_proof}
= \sum_{d\mid n} f(d) \sum_{\delta\mid  n/d} \mu(\delta) \sum_{\substack{a=1\\ (a,n)=1 \\ a\equiv 1\, \text{(mod $d\delta$)}}}^n 1.
\end{equation}

We need the following lemma.

\begin{lemma} \label{Lemma_n_d} Let $n,d\in \N$, $d\mid n$ and let $x\in \Z$. Then
\begin{equation*}
\sum_{\substack{a=1\\ (a,n)=1\\ a\equiv x \, \text{\rm (mod $d$)}}}^n 1 = 
\begin{cases} \frac{\varphi(n)}{\varphi(d)}, & \text{ if $(x,d)=1$},\\
0, & \text{ otherwise}. \end{cases}
\end{equation*}
\end{lemma}

For any $d$ and $\delta$ such that $d\mid n$ and $\delta \mid n/d$ we have $(d\delta) \mid n$. Applying
Lemma \ref{Lemma_n_d} to \eqref{sum_proof}, with $d\delta$ instead of $d$, and with $x=1$, we deduce that
\begin{equation*} 
M_f(n) = \varphi(n) \sum_{d\mid n} f(d) \sum_{\delta\mid  n/d} \frac{\mu(\delta)}{\varphi(d\delta)}
= \varphi(n) \sum_{t\mid n} \frac1{\varphi(t)} \sum_{d\delta=t} f(d) \mu(\delta).
\end{equation*}

We obtain the following result.

\begin{theorem} Let $f$ be an arbitrary arithmetic function. Then for every $n\in \N$,
\begin{equation} \label{Menon_f_convo_form}
M_f(n): = \sum_{\substack{a=1\\ (a,n)=1}}^n f((a-1,n)) = \varphi(n) \sum_{d\mid n} \frac{(\mu*f)(d)}{\varphi(d)},
\end{equation}
where $*$ denotes the Dirichlet convolution.
\end{theorem}

Now, if $f=\id$, where $\id(n)=n$ ($n\in \N$), then $\mu*\id =\varphi$, and we deduce \eqref{Menon_id}.

Lemma \ref{Lemma_n_d} is well-known in the literature, usually proved  by the inclusion-exclusion principle. See, e.g., Apostol \cite[Th.\ 5.32]{Apo1976}.
Here we use a different approach.

\begin{proof}[Proof of Lemma {\rm \ref{Lemma_n_d}}] For each term of the sum, since $(a,n)=1$, we have $(x,d)=(a,d)=1$. We can assume that
$(x,d)=1$ (otherwise the sum is empty, and equals zero). Using the property of the M\"{o}bius function, the given sum, say $T$, can be written as
\begin{equation*} 
T = \sum_{\substack{a=1\\ a\equiv x \, \text{\rm (mod $d$)} }}^n  \sum_{e \mid (a,n)}
\mu(e) = \sum_{e \mid n} \mu(e) \sum_{\substack{j=1\\ ej\equiv x \, \text{\rm (mod $d$)} }}^{n/e} 1.
\end{equation*}

Let $d$ and $e$ be such that $d\mid n$ and $e\mid n$. The linear congruence $ej\equiv x$ (mod $d$) has solutions in $j$ if and only if $(d,e)\mid x$,
equivalent to $(d,e)=1$, since $(x,d)=1$. For $(d,e)=1$ there is one solution (mod $d$), and there are $n/(de)$ solutions (mod $n/e$), since
$n/e= (n/(de))\cdot d$. Here $n/(de)$ is an integer by the condition $(d,e)=1$. We obtain
\begin{equation*}
T = \sum_{\substack{e\mid n\\(e,d)=1}} \mu(e)\cdot \frac{n}{de}=
\frac{n}{d} \sum_{\substack{e\mid n\\(e,d)=1}} \frac{\mu(e)}{e}=
\frac{n}{d} \prod_{\substack{p\mid n\\ p\nmid d}} \left(1-\frac1{p}\right)
\end{equation*}
\begin{equation*}
= \frac{n}{d} \prod_{p\mid n} \left(1-\frac1{p}\right) \prod_{p\mid d} \left(1-\frac1{p}\right)^{-1}
= \frac{n}{d} \cdot \frac{\varphi(n)}{n}\cdot \frac{d}{\varphi(d)}
= \frac{\varphi(n)}{\varphi(d)}.
\end{equation*}
\end{proof}

\subsection{Method IV: A variant of Method III} \label{Section_Method_IV}

The proof by Method III can be shortened by using first the identity $f(n)=\sum_{d\mid n} (\mu*f)(d)$. Note that in the case
$f=\id$, this is the Gauss formula $n=\sum_{d\mid n} \varphi(d)$. We deduce
\begin{equation*}
M_f(n): = \sum_{\substack{a=1\\ (a,n)=1}}^n f((a-1,n)) = \sum_{\substack{a=1\\ (a,n)=1}}^n \sum_{d\mid (a-1,n)} (\mu*f)(d) 
\end{equation*}
\begin{equation*}
=\sum_{\substack{a=1\\ (a,n)=1}}^n \sum_{\substack{d\mid a-1\\ d\mid n}} (\mu*f)(d)
= \sum_{d\mid n} (\mu*f)(d) \sum_{\substack{a=1\\  (a,n)=1\\ a\equiv 1\, \text{\rm (mod $d$)}}}^n 1.
\end{equation*}

Here the inner sum is $\varphi(n)/\varphi(d)$ by Lemma \ref{Lemma_n_d}, which proves identity \eqref{Menon_f_convo_form}. In the case
$f=\id$ we obtain \eqref{Menon_id}.

\subsection{Method V: Another variant of Method III}

The proof by Method III can also be presented in a slightly different form, by changing the order of summation. Namely,
by using first the property of the M\"{o}bius $\mu$ function,
\begin{equation*}
M_f(n): = \sum_{\substack{a=1\\ (a,n)=1}}^n f((a-1,n)) = \sum_{a=1}^n f((a-1,n)) \sum_{d\mid (a,n)} \mu(d) 
\end{equation*}
\begin{equation*}
=\sum_{a=1}^n f((a-1,n))  \sum_{d\mid a, \, d\mid n} \mu(d)
= \sum_{d\mid n} \mu(d)
\sum_{\substack{a=1\\ d\mid a}}^n f((a-1,n)).
\end{equation*}

Here we need the following lemma.

\begin{lemma} \label{Lemma_d_mid_a} If $f$ is an arithmetic function and $n,d\in \N$ such that $d\mid n$, then
\begin{equation*}
\sum_{\substack{a=1\\ d\mid a}}^n f((a-1,n)) = \frac{n}{d} \sum_{\substack{e\mid n\\(e,d)=1}} \frac{(\mu*f)(e)}{e}.
\end{equation*}
\end{lemma}

Using Lemma \ref{Lemma_d_mid_a} we deduce
\begin{equation} \label{group_unitary_convo}
M_f(n)= \sum_{d\mid n} \mu(d) \frac{n}{d} \sum_{\substack{e\mid n\\(e,d)=1}} \frac{(\mu*f)(e)}{e} =
n \sum_{e\mid n} \frac{(\mu*f)(e)}{e} \sum_{\substack{d\mid n\\(d,e)=1}} \frac{\mu(d)}{d},
\end{equation}
where the inner sum is $(\varphi(n)/n)(\varphi(e)/e)^{-1}$ (see the proof of Lemma \ref{Lemma_n_d}), giving
\begin{equation*}
M_f(n)= \varphi(n) \sum_{e\mid n} \frac{(\mu*f)(e)}{\varphi(e)},
\end{equation*}
which is identity \eqref{Menon_f_convo_form}.

\begin{proof}[Proof of Lemma {\rm \ref{Lemma_d_mid_a}}] By using that $f(n)=\sum_{d\mid n} (\mu*f)(d)$, we deduce
\begin{equation*}
U:= \sum_{\substack{a=1\\ d\mid a}}^n f((a-1,n)) = \sum_{j=1}^{n/d} f((jd-1,n)) = \sum_{j=1}^{n/d} \sum_{e\mid jd-1, \, e\mid n} (\mu*f)(e)
\end{equation*}
\begin{equation*}
= \sum_{e\mid n} (\mu*f)(e) \sum_{\substack{j=1\\ jd\equiv 1\, \text{(mod $e$)}}}^{n/d} 1,
\end{equation*}
where the inner sum is, see the proof of Lemma \ref{Lemma_n_d}, $n/(de)$ if $(d,e)=1$ and $0$ otherwise. This gives
\begin{equation*}
U = \sum_{\substack{e\mid n\\(e,d)=1}} (\mu*f)(e)\cdot \frac{n}{de}=
\frac{n}{d} \sum_{\substack{e\mid n\\(e,d)=1}} \frac{(\mu*f)(e)}{e}.
\end{equation*}
\end{proof}

We remark that by grouping the terms in \eqref{group_unitary_convo} according to the values $de=\delta$ we have
\begin{equation*}
M_f(n)= n \sum_{\delta \mid n} \frac1{\delta} \sum_{\substack{de=\delta \\(d,e)=1}} \mu(d) (\mu*f)(e)
= \sum_{\delta \mid n} \frac{n}{\delta} (\mu \times (\mu *f))(\delta)
\end{equation*}
\begin{equation} \label{Menon_convo}
= (\id * (\mu \times (\mu *f)))(n),
\end{equation}
where $\times$ is the unitary convolution. If $f$ is multiplicative, then it follows that $M_f$ is also multiplicative, and  for every prime power $p^\nu$ ($\nu \ge 1$),
\begin{equation*}
M_f(p^\nu)= (\id * (\mu \times (\mu *f)))(p^\nu) = \sum_{j=0}^\nu p^{\nu-j} (\mu \times (\mu *f))(p^j)
\end{equation*}
\begin{equation*}
=  p^\nu + \sum_{j=1}^\nu p^{\nu-j} (\mu(p^j)+ (\mu *f)(p^j)) =  \sum_{j=0}^\nu p^{\nu-j} (\mu *f)(p^j) - p^{\nu-1}
\end{equation*}
\begin{equation*}
= (\id * \mu * f)(p^\nu) -p^{\nu-1} = (\varphi*f)(p^\nu)-p^{\nu-1},
\end{equation*}
giving the identity
\begin{equation} \label{product_form_f}
M_f(n)= \prod_{p^\nu \mid \mid n} \left((\varphi*f)(p^\nu)-p^{\nu-1}\right).
\end{equation}

\subsection{Method VI: Proof by the orbit counting lemma} \label{Section_proof_orbit_counting}

The orbit counting lemma or Cauchy-Frobenius-Burnside lemma can be stated as follows. See, e.g., Bogart \cite{Bog1991}, Neumann \cite{Neu1979},
Rotman \cite[Ch.\ 2]{Rot2005}, Wright \cite{Wri1981} for its proof, history and applications.

\begin{lemma} Let $G$ be a finite group and $X$ be a finite set. Let $\psi:G\times X\to X$, $\psi(g,x)=gx$ \textup{($g\in G$, $x\in X$)} be an action of $G$ on 
the set $X$. Let $X/G$ be the set of orbits and let $X^g = \{x \in X: gx = x\}$ be the set of elements $x\in X$ fixed by $g\in G$. Then
\begin{equation*}
|X/G|=\frac1{|G|} \sum_{g\in G } |X^g|.
\end{equation*}
\end{lemma}

Apply this lemma in the case of $G=U(\Z_n)=\{k \text{ (mod $n$)}: (k,n)=1\}$, the group of units of the ring $\Z_n$ of residue classes (mod $n$), which is a group of order $\varphi(n)$, and
$X=\{1,2,\ldots,n\}$. For any $k\in U(\Z_n)$ and $x\in  \{1,2,\ldots,n\}$ let $\psi(k,x)=\ell$, where $kx\equiv \ell$ (mod $n$), which is a group action.
Here $X^k=\{x: 1\le x\le n, kx\equiv x \text{ (mod $n$)} \}= \{x: 1\le x\le n, (k-1)x\equiv 0 \text{ (mod $n$)} \}$. Note that 
$|X^k|=(k-1,n)$.

To determine the number of orbits, we need the following lemma.

\begin{lemma} \label{Lemma_cong}
Let $a,b\in \Z$ and $n\in \N$ be given integers. The congruence $ax \equiv b$ \textup{(mod $n$)} has solutions $x$ such that $(x,n)=1$ if and only
if $(a,n)=(b,n)$.
\end{lemma}

This lemma shows that the orbits of the considered group action are $O_d=\{k: 1\le k\le n, (k,n)=d \}$, where $d$ is a divisor of $n$. Therefore,
the number of orbits is $|X/G|=\tau(n)$. This completes the proof of identity \eqref{Menon_id}.

\begin{proof}[Proof of Lemma {\rm \ref{Lemma_cong}}]
Assume that there is a solution $x=t$ satisfying $at \equiv b$ (mod $n$) and $(t,n)=1$. Then $(at,n)=(b,n)$.
Since $(t,n)=1$, we obtain that $(a,n)=(b,n)$.

Conversely, suppose that $(a,n)=(b,n)=d$.  Let us denote $A=a/d$, $B=b/d$, $N=n/d$. Then $(A,N)=(B,N)=1$. Hence the linear congruence
$Ax\equiv B$ (mod $N$) has a solution $x=x_0$ (it has, in fact, one solution $x$ modulo $N$) satisfying $Ax_0\equiv B$ (mod $N$). Note that
$(Ax_0,N)=(B,N)=1$, hence $(x_0,N)=1$. It follows that the solutions of the original congruence $ax \equiv b$ (mod $n$)
are $x_j=x_0+jN$, where $0\le j\le d-1$. We have to show that there exist solutions $x_j$ such that $(x_j,n)=1$.

Let $C=\prod_{p\mid d,\; p\nmid N} p$, where $(C,N)=1$. By the Chinese remainder theorem, for the simultaneous congruences
$y\equiv x_0$ (mod $N$), $y\equiv 1$ (mod $C$) there is a solution $y=y_0$ satisfying $y_0\equiv x_0$ (mod $N$), $y_0\equiv 1$ (mod $C$).
Hence $(y_0,N)=(x_0,N)=1$ and $(y_0,C)=1$.

This gives $(y_0,n)=1$ and $Ay_0\equiv Ax_0\equiv B$ (mod $N$). Hence $ay_0\equiv b$ (mod $n$), and $y_0$ is a solution coprime to $n$.
\end{proof}

It can be proved that in the case $(a,n)=(b,n)=d$, the number of incongruent solutions of the congruence
$ax \equiv b$ (mod $n$) is
\begin{equation} \label{number_sol}
\frac{\varphi(n)}{\varphi(n/d)}= d \prod_{\substack{p\mid n\\ p\nmid \, n/d}} \left(1-\frac1{p} \right).
\end{equation}

See Bibak et al. \cite[Th.\ 3.1]{BibTot2017}, Gro\v{s}ek and Porubsk\'{y} \cite[Th.\ 2]{GroPor2013} for the use of two different
approaches to deduce \eqref{number_sol}. Also see paper Garcia and Ligh \cite{GarLig1983}, defining a generalization of
Euler's totient function for arithmetic progressions, and deducing the same formula \eqref{number_sol}.

\subsection{Method VII: Proof by exponential sums} \label{Section_Method_VII}

We have by denoting $a-1=b$,
\begin{equation*} 
M(n):= \sum_{\substack{a \text{ (mod $n$)} \\ (a,n)=1}} (a-1,n)= \sum_{\substack{a,b \text{ (mod $n$)} \\ a-1\equiv b \text{ (mod $n$)} \\ (a,n)=1}} (b,n). 
\end{equation*}

By the familiar orthogonality property
\begin{equation} \label{sum_exp}
\frac1{n} \sum_{j \text{ (mod $n$)}} e(kj/n) = \begin{cases} 1, & \text{ if $n\mid k$}, \\ 0,  & \text{ otherwise},
\end{cases}    
\end{equation}
we deduce that
\begin{equation*} 
M(n)= \frac1{n} \sum_{\substack{a,b \text{ (mod $n$)} \\ (a,n)=1}} (b,n) \sum_{j \text{ (mod $n$)}} e((1-a+b)j/n) 
\end{equation*}
\begin{equation}  \label{M_exp_1}
= \frac1{n} \sum_{j \text{ (mod $n$)}} e(j/n)  \sum_{b \text{ (mod $n$)}} (b,n) e(bj/n) \sum_{\substack{ a \text{ (mod $n$)}\\ (a,n)=1}} e(-aj/n).
\end{equation}

Here the last sum is Ramanujan's sum $c_n(-j)=c_n(j)$, and we use its H\"older's evaluation, namely
\begin{equation} \label{Holder}
c_n(j)= \frac{\varphi(n)\mu(n/(j,n))}{\varphi(n/(j,n))} \quad (j,n\in \N).
\end{equation}

We also need the next result. See T\'oth \cite[Prop.\ 7]{Tot2011Weigh}.
\begin{lemma} \label{Lemma_gcd_exp} For every $n\in \N$ and $k\in \Z$,
\begin{equation} \label{form_gcd_exp}
\sum_{a=1}^n (a,n) e(ka/n)= \sum_{d\mid (k,n)} d \varphi(n/d).
\end{equation}
\end{lemma}

According to \eqref{M_exp_1}, \eqref{Holder} and \eqref{form_gcd_exp},
\begin{equation*} 
M(n) = \frac{\varphi(n)}{n} \sum_{j \text{ (mod $n$)}} e(j/n)  \sum_{d\mid (j,n)} d\varphi(n/d)  \frac{\mu(n/(j,n))}{\varphi(n/(j,n))},
\end{equation*}
and by grouping the terms of the last sum according to the values $(j,n)=\delta$ with $j= \delta t$, $(t,n/d)=1$ we have
\begin{equation*} 
M(n) = \frac{\varphi(n)}{n} \sum_{\delta \mid n} \frac{\mu(n/\delta)}{\varphi(n/\delta)} \sum_{d\mid \delta} d\varphi(n/d)   
\sum_{\substack{t \text{ (mod $n/\delta$)}\\ (t,n/\delta)=1}} e(t/(n/\delta)),
\end{equation*}
the last sum being $c_{n/\delta}(1)=\mu(n/\delta)$ by \eqref{Holder}. This gives
\begin{equation} \label{sum_Landau}
M(n) = \frac{\varphi(n)}{n} \sum_{\delta \mid n} \frac{\mu^2(n/\delta)}{\varphi(n/\delta)} \sum_{d\mid \delta} d\varphi(n/d)
= \frac{\varphi(n)}{n} \sum_{dek=n} \frac{\mu^2(e)}{\varphi(e)} d\varphi(ke)
\end{equation}
\begin{equation*}
= \frac{\varphi(n)}{n} \sum_{dm=n} d\varphi(m) \sum_{ke=m} \frac{\mu^2(e)}{\varphi(e)},
\end{equation*}
where the inner sum is $m/\varphi(m)$ (Landau's identity), and obtain
\begin{equation*}
M(n)=\frac{\varphi(n)}{n} \sum_{dm=n} dm = \varphi(n) \sum_{dm=n} 1= \varphi(n)\tau(n).
\end{equation*}

\begin{proof}[Proof of Lemma {\rm \ref{Lemma_gcd_exp}}] Using the Gauss formula $n=\sum_{d\mid n} \varphi(d)$ we have
\begin{equation*}
\sum_{a=1}^n (a,n) e(ka/n)=  \sum_{a=1}^n e(ka/n) \sum_{d\mid (a,n)} \varphi(d) = \sum_{d\mid n} \varphi(d) 
\sum_{\substack{a=1\\ d\mid a}}^n e(ka/n)
\end{equation*}
\begin{equation*}
= \sum_{d\mid n} \varphi(d) \sum_{b=1}^{n/d} e(kb/(n/d))= \sum_{\substack{d\mid n\\ n/d \, \mid k}} \frac{n}{d} \varphi(d) = 
\sum_{d\mid (k,n)} d\varphi(n/d),
\end{equation*}
applying \eqref{sum_exp}. 
\end{proof}

\subsection{Method VIII: Another proof by exponential sums} \label{Section_Method_VIII}

Observe that 
\begin{equation} \label{M_a_b}
M(n):= \sum_{\substack{a=1\\ (a,n)=1}}^n (a-1,n)= \sum_{\substack{b=1\\ (b,n)=1}}^n \sum_{\substack{a=1\\ ab\equiv 1 
\text{ (mod $n$)}}}^n (a-1,n),    
\end{equation}
since for every fixed $b$ with $1\le b\le n$ and $(b,n)=1$ there is exactly one integer $a$ such that $1\le a\le n$ and 
$ab\equiv 1$ (mod $n$), whence $(a,n)=1$; and for different values of $b$ the corresponding values of $a$ are different. Note that the condition $(b,n)=1$ can be removed in \eqref{M_a_b}, but we keep it and use it in what follows.

By property \eqref{sum_exp} we deduce that
\begin{equation*}
M(n)= \frac1{n} \sum_{\substack{b=1\\ (b,n)=1}}^n \sum_{a=1}^n (a-1,n) \sum_{j=1}^n e((ab-1)j/n))
\end{equation*}
\begin{equation*}
= \frac1{n} \sum_{j=1}^n e(-j/n) \sum_{\substack{b=1\\ (b,n)=1}}^n \sum_{a=1}^n (a-1,n) e(abj/n)).
\end{equation*}

According to \eqref{form_gcd_exp},
\begin{equation*}
\sum_{a=1}^n (a-1,n) e(abj/n)) = \sum_{a=1}^n (a,n) e((a+1) bj/n))
\end{equation*}
\begin{equation*}
= e(bj/n) \sum_{d\mid (bj,n)} d\varphi(n/d) = e(bj/n) \sum_{d\mid (j,n)} d\varphi(n/d),  
\end{equation*}
with $(b,n)=1$. This gives
\begin{equation*}
M(n)= \frac1{n} \sum_{d\mid n} d \varphi(n/d) \sum_{\substack{b=1\\ (b,n)=1}}^n \sum_{t=1}^{n/d} e((b-1)t/(n/d)),   
\end{equation*}
where, by \eqref{sum_exp} again, the last sum is $n/d$ if $b\equiv 1$ (mod $n/d$) and $0$ otherwise. Hence
\begin{equation*}
M(n) = \sum_{d\mid n} \varphi(n/d) \sum_{\substack{b=1\\ (b,n)=1\\ b\equiv 1 \text{ (mod $n/d$)}} }^n 1 
= \sum_{d\mid n} \varphi(n/d) \frac{\varphi(n)}{\varphi(n/d)}= \varphi(n)\tau(n), 
\end{equation*}
by using Lemma \ref{Lemma_n_d}. Note that Ramanujan sums are not needed here.

One may wonder whether this method works by using character sums instead of exponential sums.
Well, using the orthogonality property
\begin{equation} \label{charct_orthog}
\frac1{\varphi(n)} \sum_{\chi} \chi(u)\overline{\chi}(v) = \begin{cases} 1, & \text{ if $u\equiv v$ (mod $n$)}, \\ 0,  & \text{ otherwise},
\end{cases}    
\end{equation}
where the sum is over the Dirichlet characters $\chi$ (mod $n$), from \eqref{M_a_b} we have
\begin{equation*}
M(n)= \frac1{\varphi(n)} \sum_{b=1}^n \sum_{a=1}^n (a-1,n) \sum_{\chi} \chi(ab)\overline{\chi}(1)
\end{equation*}
\begin{equation} \label{sum_not}
= \frac1{\varphi(n)} \sum_{\chi} \left(\sum_{b=1}^n \chi(b)\right) \left( \sum_{a=1}^n (a-1,n) \chi(a)\right).
\end{equation}

However, as well-known, 
\begin{equation} \label{sum_values_charact}
\sum_{b=1}^n \chi(b) = \begin{cases} \varphi(n), & \text{ if $\chi=\chi_0$, the principal character}, 
\\ 0,  & \text{ otherwise},
\end{cases}    
\end{equation}
therefore \eqref{sum_not} gives nothing but the definition of the sum $M(n)$. On the other hand, the sum
\begin{equation*} 
 \sum_{a=1}^n (a-1,n) \chi(a)
\end{equation*}
can be computed for an arbitrary Dirichlet character (mod $n$), and it gives a generalization of Menon's identity. See 
Section \ref{Sect_id_Dirichlet_char}.

\subsection{Method IX: Proof via the finite Fourier representation of $n$-even functions} \label{Section_n_even}

A function $f:\N \to \C$ is said to be an $n$-even function -- or even function (mod $n$) -- if
$f(k) = f((k,n))$ for all $k\in \N$, where $n\in \N$ is fixed and $(k,n)$ is the gcd of $k$ and $n$. Note that if $f$ is $n$-even, then
$f$ is $n$-periodic, and by this periodicity it can be extended to a function defined on
$\Z$. For example, the function $k\mapsto (k,n)$ is $n$-even. More generally, $k\mapsto F((k,n))$ is $n$-even, where $F$ is
an arbitrary arithmetic function. Another example of an $n$-even function is $k\mapsto c_n(k)$, where $c_n(k)$ is the Ramanujan sum, defined by
\eqref{Ramanujan_sum}, and evaluated as
\begin{equation*} 
c_n(k)=\sum_{d \mid (k,n)} d \mu(n/d) \quad (k, n \in \N).
\end{equation*}

The set ${\cal E}_n$ of $n$-even functions forms a $\tau(n)$ dimensional linear space, and the Ramanujan sums $k\mapsto c_d(k)$ with $d\mid n$
form a basis of the space ${\cal E}_n$. Consequently, every $n$-even function $f$ has a (Ramanujan-)Fourier expansion of the form
\begin{equation} \label{Ramanujan_Fourier_exp}
f(k)= \sum_{d\mid n} \alpha_f(d) c_d(k) \quad (k\in \N),
\end{equation}
where the (Ramanujan-)Fourier coefficients $\alpha_f(d)$ ($d\mid n$) of $f$ are uniquely determined and given by
\begin{equation*}
\alpha_f(d)= \frac1{n} \sum_{e\mid n} f(e)c_{n/e}(n/d)
\end{equation*}
\begin{equation} \label{Ramanujan_Fourier_coeff}
= \frac1{n} \sum_{e\mid n/d} e (\mu*f)(n/e).
\end{equation}

In order to deduce a Menon-type identity, related to the finite Fourier expansion of $n$-even functions, we need the following further 
preliminaries. 

The Cauchy convolution of the $n$-periodic functions $f$ and $g$ is defined by
\begin{equation*} 
(f \otimes g)(k)= \sum_{a+b\equiv k \text{ (mod $n$)}} f(a)g(b)  \quad (k\in \N).
\end{equation*}

For example, the Cauchy convolution of the exponential functions $f(a)=  e(as/n)$ and
$g(b)= e(bt/n)$ is
\begin{equation} \label{Cauchy_exp}
(f \otimes g)(k)= \sum_{a+b\equiv k \text{ (mod $n$) }} e(as/n) e(bt/n) =
\begin{cases} n e(ks/n), &\text{ if $s\equiv t$ (mod $n$)},\\
         0, &\text{ otherwise}.
\end{cases}
\end{equation}

We need the following lemmas.

\begin{lemma} \label{Lemma_Cauchy_Ramanujan} If $d\mid n$ and
$e\mid n$, then the Cauchy convolution of the Ramanujan sums $a\mapsto c_d(a)$ and $b\mapsto c_e(b)$ is
\begin{equation} \label{Cauchy_Raman}
\sum_{a+b\equiv k \text{ (mod $n$) }} c_d(a)c_e(b)= \begin{cases} n c_d(k), &\text{ if $d=e$}, \\
         0, &\text{ otherwise}.
\end{cases}
\end{equation}
\end{lemma}

\begin{proof}[Proof of Lemma {\rm \ref{Lemma_Cauchy_Ramanujan}}]
Let $n=dd_1= ee_1$. Then
\begin{equation*}
\sum_{a+b\equiv k \text{ (mod $n$) }} c_d(a)c_e(b)= \sum_{a+b \equiv k \text{ (mod $n$) }} \sum_{\substack{
j=1\\(j,d)=1}}^d e(aj/d) \sum_{\substack{\ell=1\\(\ell,e)=1}}^e e(\ell b/e)
\end{equation*}
\begin{equation*}
= \sum_{\substack{j=1\\(j,d)=1}}^d \sum_{\substack{\ell=1\\(\ell,e)=1}}^e \sum_{a+b \equiv k \text{ (mod $n$) }} e(ajd_1/n) e(b\ell e_1/n)
\end{equation*}
\begin{equation*}
= \sum_{\substack{j=1\\(j,d)=1}}^d \sum_{\substack{\ell=1\\(\ell,e)=1}}^e \sum_{jd_1\equiv \ell e_1 \text{ (mod $n$) }} n e(kjd_1/n),
\end{equation*}
according to \eqref{Cauchy_exp}. Here $1\le d_1\le jd_1\le dd_1=n$, $1\le e_1\le \ell e_1\le ee_1=n$, hence $jd_1\equiv \ell e_1$ (mod $n$) holds
if and only if $jd_1=\ell e_1$, that is $jd_1de=\ell e_1de$, $jne=\ell dn$, $je=\ell d$. Taking into account that $(j,d)=1$, $(\ell,e)=1$, we deduce
that $j=\ell$ and $d=e$. Hence, if $d\ne e $, then the sum is zero (empty) and if $d=e$, then it is
\begin{equation*}
n \sum_{\substack{j=1\\(j,d)=1}}^d e(jk/d)= n c_d(k),
\end{equation*}
completing the proof.
\end{proof}

\begin{lemma} \label{Lemma_Cauchy_prod} If $f$ and $g$ are $n$-even functions, then the Cauchy convolution $f \otimes g$ is also $n$-even, and its
Fourier coefficients are
\begin{equation} \label{Cauchy_Fourier_coeff}
\alpha_{f\otimes g}(d)= n\, \alpha_f(d) \alpha_g(d) \quad (d\mid n).
\end{equation}
\end{lemma}

\begin{proof}[Proof of Lemma {\rm \ref{Lemma_Cauchy_prod}}]
We have
\begin{equation*}
(f\otimes g)(k) = \sum_{a+b\equiv k \text{ (mod $n$)}} f(a)g(b) 
\end{equation*}
\begin{equation*}
=\sum_{a+b \equiv k \text{ (mod $n$) }} \sum_{d\mid n} \alpha_f(d) c_d(a)
\sum_{e\mid n} \alpha_g(e)c_e(b)
\end{equation*}
\begin{equation*}
= \sum_{d\mid n} \sum_{e\mid n} \alpha_f(d) \alpha_g(e) \sum_{a+b \equiv k \text{ (mod $n$) }} c_d(a) c_e(b) = \sum_{d\mid n} n \alpha_f(d)
\alpha_g(d) c_d(k),
\end{equation*}
by using \eqref{Cauchy_Raman}. Since for $d\mid n$ the function $k\mapsto c_d(k)$ is $n$-even we deduce that $f\otimes g$ is also $n$-even, and its
Fourier coefficients are $n\alpha_f(d)\alpha_g(d)$.
\end{proof}

See Haukkanen \cite{Hau2001}, McCarthy \cite[Ch.\ 2]{McC1986}, Montgomery \cite[Ch.\ 2]{Mon2014}, T\'oth and Haukkanen \cite{TotHauk2011}
for more details on these notions and properties.

Now, for the proof of a Menon-type identity let $f$ be an arbitrary $n$-even function with Fourier coefficients $\alpha_f(d)$ ($d\mid n$).
Consider the sum
\begin{equation*}
S_f(k,n):= \sum_{\substack{b \text{ (mod $n$)} \\ (b,n)=1}} f(k-b).
\end{equation*}

If $f(b)=(b,n)$ ($b\in \N$) is the gcd function and $k=1$, then $S_f(1,n)$ is exactly the sum appearing in Menon's identity.

For any $n$-even function $f$ and any $k\in \N$ the sum $S_f(k,n)$ can be written as
\begin{equation} \label{eval_S_f}
S_f(k,n) = \sum_{\substack{a+b\equiv k \text{ (mod $n$)}\\ (b,n)=1}} f(a) = \sum_{a+b\equiv k \text{(mod $n$)}} f(a) \delta(b,n),
\end{equation}
which is the Cauchy convolution of the function $f$ and the Kronecker $\delta$, given here by $\delta(b,n)=1$ if $(b,n)=1$, and
$\delta(b,n)=0$ if $(b,n)>1$. Note that $\delta$ is $n$-even.

The main observation, which leads to the proof of a Menon-type identity is that the function $k\mapsto S_f(k,n)$ is also $n$-even, being
the Cauchy convolution of $n$-even functions, by Lemma \ref{Lemma_Cauchy_prod}.

The Fourier coefficients of $\delta$ are
\begin{equation} \label{alpha_delta}
\alpha_{\delta}(d) =\frac1{n} c_n(n/d)= \frac{\varphi(n) \mu(d)}{n\varphi(d)} \quad (d\mid n),
\end{equation}
the second equality being a consequence of H\"older's evaluation \eqref{Holder} of the Ramanujan sums.

Now, \eqref{eval_S_f} and \eqref{Cauchy_Fourier_coeff} imply that
\begin{equation*}
S_f(k,n)= n \sum_{d\mid n}  \alpha_f(d)\alpha_{\delta}(d) c_d(k).
\end{equation*}

By using \eqref{alpha_delta} we deduce the following general identity, due to Cohen \cite{Coh1959II} in an implicit form. See McCarthy \cite[p.\ 85]{McC1986}.  

\begin{theorem} \label{Th_gen_f_even} Let $f$ be an arbitrary $n$-even function with Fourier coefficients $\alpha_f(d)$ \textup{($d\mid n$)}. Then
\begin{equation*}
S_f(k,n):= \sum_{\substack{b \textup{ (mod $n$)} \\ (b,n)=1}} f(k-b) = \varphi(n) \sum_{d\mid n} \alpha_f(d)
\frac{\mu(d)}{\varphi(d)} c_d(k).
\end{equation*}
\end{theorem}

In fact, Cohen \cite[Th.\ 1]{Coh1959II} stated the result of Lemma \ref{Lemma_Cauchy_prod}, and then presented a series of corollaries, 
including formulas for the sums $S_f(k,n)$ of above with special functions $f$, e.g., $f(b)=\tau((b,n)), \sigma((b,n)), c_n(b), (n,b)^s$, 
where $c_n(b)$ is the Ramanujan sum, $s$ is real. See \cite[Cor.\  2.2, 2.3, 7.2, 17.1]{Coh1959II}.

Now we show how Menon's identity is deduced from Theorem \ref{Th_gen_f_even}. Consider the function $f(b)=(b,n)$. If $a\mid n$, then we have
\begin{equation*}
(\mu*f)(a)= \sum_{d\mid a} f(d) \mu(a/d)= \sum_{d\mid a} d \mu(a/d)= \varphi(a).
\end{equation*}

Therefore, by using \eqref{Ramanujan_Fourier_coeff} we deduce that the Fourier coefficients $\alpha(d)$ of the function $f(b)=(b,n)$ are 
\begin{equation*}
\alpha(d)= \frac1{n} \sum_{e\mid n/d} e \varphi(n/e).
\end{equation*}

We obtain that
\begin{equation*}
S(k,n):= \sum_{\substack{b \text{ (mod $n$)} \\ (b,n)=1}} (k-b,n) =  \frac{\varphi(n)}{n} \sum_{d\mid n} \frac{\mu(d)}{\varphi(d)} c_d(k)
\sum_{e\mid n/d} e \varphi(n/e).
\end{equation*}

Now let $k=1$. Then $c_d(1)= \mu(d)$ and with the notation $n=dt$, $t=n/d= ej$, 
\begin{equation*}
S(1,n)= \frac{\varphi(n)}{n} \sum_{d\mid n} \frac{\mu^2(d)}{\varphi(d)} \sum_{e\mid n/d} e \varphi(n/e) = \frac{\varphi(n)}{n}
\sum_{dej=n} \frac{\mu^2(d)}{\varphi(d)} e \varphi(dj), 
\end{equation*}
which is identical to \eqref{sum_Landau}, and gives $\varphi(n)\tau(n)$.

\subsection{Method X: Proof by Vaidyanathaswamy's class division of integers (mod $n$)} \label{Section_Vaidya_class_div}

Group the elements $a$ of the set $S=\{1,2,\ldots,n\}$ according to the values $(a,n)=d$. Let $S_d=\{a\in S: (a,n)=d\}$, where $d\mid n$. It is clear
that $S=\cup_{d\mid n} S_d$ and $|S_d|=\varphi(n/d)$ for any $d\mid n$. This leads to the Gauss formula $n=\sum_{d\mid n} \varphi(n/d)=
\sum_{d\mid n} \varphi(d)$.

Vaidyanathaswamy \cite{Vai1937} defined the sum of the subsets $S_{d_1}$ and $S_{d_2}$ ($d_1\mid n$, $d_2\mid n$), denoted by $S_{d_1} \oplus S_{d_2}$,
as the multiset
\begin{equation*}
S_{d_1} \oplus S_{d_2} = \{ x + y \text{ (mod $n$)}: x\in S_{d_1}, y\in S_{d_2}\}.
\end{equation*}

The main related property proved in paper \cite{Vai1937} is that in $S_{d_1} \oplus S_{d_2}$ all elements of a class $S_d$ occur the same number
of times, denoted by $\gamma(d_1,d_2,d)$, which depends only on $d_1$, $d_2$ and $d$. The coefficients $\gamma(d_1,d_2,d)$ can be evaluated in terms of
the Ramanujan sums, which is a result of Ramanathan \cite{Ram1944}.

\begin{lemma} \label{Lemma_coeff_Ramanujan} 
1) For any $d_1\mid n$ and $d_2\mid n$ one has
\begin{equation*}
S_{d_1} \oplus S_{d_2} = \bigcup_{d\mid n} \gamma(d_1,d_2,d) S_d,
\end{equation*}
where $\gamma(d_1,d_2,d) S_d$ means the multiset where each element occurs
\begin{equation} \label{gamma_d1_d2}
\gamma(d_1,d_2,d) = \frac1{n} \sum_{\delta \mid n} c_{n/d_1}(\delta) c_{n/d_2}(\delta) c_{n/\delta}(d).
\end{equation}
times. 

2) In the special case $d_1=d_2=1$ we have for any $d\mid n$,
\begin{equation} \label{gamma_1_1_d}
\gamma(1,1,d) = \varphi(n) \prod_{\substack{p\mid n\\ p\nmid d}} \left(1-\frac1{p-1} \right).
\end{equation}
\end{lemma}

Now, for the proof of Menon's identity we follow the arguments given by Menon \cite{Men1965}, and explain some more details by using properties of 
$n$-even functions. Note that Menon \cite{Men1965} did not use the concept of $n$-even functions, which has been introduced and investigated by Cohen in a series of papers, including \cite{Coh1955,Coh1958I,Coh1959II}.

Let $f$ be an arithmetic function, let $d_1\mid n$, $d_2\mid n$ and consider the sum
\begin{equation*}
F(n,d_1,d_2):= \sum_{\substack{a \text{ (mod $n$)} \\ (a,n)=d_1}} f((a+d_2,n)).
\end{equation*}

The following property is needed. See below its proof.
\begin{lemma} \label{Lemma_P}
The sum $F(n,d_1,d_2)$ is unchanged if $d_2$ is replaced by any integer $b$ such that $(b,n)=d_2$.
\end{lemma}

By Lemma \ref{Lemma_P},
\begin{equation} \label{F_n_d}
F(n,d_1,d_2)= \sum_{\substack{a \text{ (mod $n$)} \\ (a,n)=d_1}} f((a+b,n)), \quad \text{ where $(b,n)=d_2$}.
\end{equation}

If we sum equations \eqref{F_n_d} over the values of $b$, note that there are $\varphi(n/d_2)$ such values of $b$ (mod $n$), 
we obtain
\begin{equation*}
F(n,d_1,d_2)= \frac1{\varphi(n/d_2)} \sum_{\substack{a,b \text{ (mod $n$)} \\ (a,n)=d_1\\ (b,n)=d_2}} f((a+b,n)),
\end{equation*}
and grouping the terms according to the values $a+b=c$ we have
\begin{equation*}
F(n,d_1,d_2)= \frac1{\varphi(n/d_2)} \sum_{c \text{ (mod $n$)}} f((c,n)) \sum_{\substack{a,b \text{ (mod $n$)} \\ (a,n)=d_1\\ (b,n)=d_2\\
a+b\equiv c \text{ (mod $n$}) }} 1.
\end{equation*}

Here the inner sum is exactly the Cauchy convolution $(\varrho_1 \otimes \varrho_2)(c)$ of the functions $\varrho_1$ and $\varrho_2$, where
\begin{equation} \label{funct_varrho}
  \varrho_j(a)=\begin{cases} 1, & \text{ if $(a,n)=d_j$}, \\ 0, & \text{ otherwise},
  \end{cases} \quad (j=1,2).
\end{equation}

Since $\varrho_1$ and $\varrho_2$ are $n$-even, $\varrho_1 \otimes \varrho_2$ is also $n$-even by Lemma \ref{Lemma_Cauchy_prod}, that is
$(\varrho_1 \otimes \varrho_2)(c)= (\varrho_1 \otimes \varrho_2)((c,n))$. We have
\begin{equation*}
F(n,d_1,d_2)= \frac1{\varphi(n/d_2)} \sum_{c \text{ (mod $n$)}} f((c,n)) (\varrho_1 \otimes \varrho_2)((c,n))
\end{equation*}
\begin{equation*}
= \frac1{\varphi(n/d_2)} \sum_{d\mid n} f(d) \varphi(n/d) (\varrho_1 \otimes \varrho_2)(d),
\end{equation*}
by grouping the terms according to the values $(c,n)=d$. Observe that $(\varrho_1 \otimes \varrho_2)(d)=\gamma(d_1,d_2,d)$, given by
\eqref{gamma_1_1_d}. This leads to the following general result of Menon \cite[Eq.\ (4.5)]{Men1965}.

\begin{theorem} Let $f$ be an arithmetic function, let $n\in \N$ and $d_1,d_2\mid n$. Then
\begin{equation*}
\sum_{\substack{a\, \textup{(mod $n$)} \\ (a,n)=d_1}} f((a+d_2,n))=
\frac1{\varphi(n/d_2)} \sum_{d\mid n} f(d) \gamma(d_1,d_2,d) \varphi(n/d).
\end{equation*}
\end{theorem}

If $d_1=d_2=1$, then by \eqref{gamma_1_1_d} we deduce the next formula, given by Menon \cite[Eq.\ (4.8)]{Men1965}.

\begin{corollary} Let $f$ be an arithmetic function and $n\in \N$. Then
\begin{equation} \label{id_by_Vaidya_method}
\sum_{\substack{a\, \textup{(mod $n$)} \\ (a,n)=1}} f((a-1,n))=
\sum_{d\mid n} f(d) \varphi(n/d) \prod_{\substack{p\mid n\\ p\nmid d}} \left(1-\frac1{p-1} \right).
\end{equation}
\end{corollary}

We remark that it can be showed by direct computations that
\begin{equation} \label{varphi_n_d}
\varphi(n/d) \prod_{\substack{p\mid n\\ p\nmid d}}
\left(1-\frac1{p-1} \right) = \varphi(n) \sum_{\delta \mid n/d} \frac{\mu(\delta)}{\varphi(d\delta)},
\end{equation}
and that for any function $f$,
\begin{equation} \label{sum_f}
\sum_{d\mid n} f(d) \sum_{\delta \mid n/d} \frac{\mu(\delta)}{\varphi(d\delta)} = \sum_{d\mid n} \frac{(\mu*f)(d)}{\varphi(d)}, 
\end{equation}
therefore \eqref{id_by_Vaidya_method}, \eqref{varphi_n_d} and \eqref{sum_f} imply formula \eqref{Menon_f_convo_form}, which reduces to Menon's original identity in the case $f=\id$.

\begin{proof}[Proof of Lemma {\rm \ref{Lemma_coeff_Ramanujan}}] 1) According to its definition,
\begin{equation*}
\gamma(d_1,d_2,d) = \sum_{\substack{a+b\equiv d \text{ (mod $n$)}\\ (a,n)=d_1\\ (b,n)=d_2}} 1 = (\varrho_1 \otimes \varrho_2)(d),
\end{equation*}
representing the Cauchy convolution of the functions $\varrho_1$ and $\varrho_2$, given by \eqref{funct_varrho}. As mentioned above, the function 
$\varrho_1 \otimes \varrho_2$ is $n$-even. Also, the Fourier coefficients of $\varrho_j$ ($j=1,2$) are
\begin{equation*}
\alpha_{\varrho_j}(d)=\frac1{n} \sum_{e\mid n} \varrho_j(e)c_{n/e}(n/d) = \frac1{n} c_{n/d_j}(n/d) \quad (d\mid n).
\end{equation*}

Therefore, by Lemma \ref{Lemma_Cauchy_prod}, the Fourier coefficients of $\varrho_1 \otimes \varrho_2$ are
\begin{equation*}
\alpha_{\varrho_1\otimes \varrho_2}(d)=  \frac1{n} c_{n/d_1}(n/d) c_{n/d_2}(n/d) \quad (d\mid n),
\end{equation*}
and from the Fourier representation \eqref{Ramanujan_Fourier_exp} we obtain identity \eqref{gamma_d1_d2}.

2) If $d_1=d_2=1$, then \eqref{gamma_d1_d2} implies \eqref{gamma_1_1_d} by direct computations using properties of 
the Ramanujan sums.
\end{proof}

Let $d_1, d_2,d$ be divisors of $n$. It can be shown that $\gamma(d_1,d_2,d)=0$, unless $(d_1,d_2)=(d_1,d)=(d_2,d)=\delta$, when
\begin{equation*}
\gamma(d_1,d_2,d)= \frac{\varphi(n/d_1)\varphi(n/d_2)}{\varphi(n/d)} \prod_{\substack{p\mid n/\delta \\ p\nmid d/\delta\\ p\nmid d_1/\delta \\
p\nmid d_2/\delta}} \left(1-\frac1{p-1} \right).
\end{equation*}

For $d_1=d_2=1$ this recovers formula \eqref{gamma_1_1_d}. See Haukkanen \cite{Hau2000} and Sivaramakrishnan \cite[Ch. XV]{Siv1989} for more details.

\begin{proof}[Proof of Lemma {\rm \ref{Lemma_P}}] 
Consider the functions $g(a)=f((a,n))$ and $\varrho_{d_1}(a)$, defined by \eqref{funct_varrho}, which are $n$-even
functions. Also
\begin{equation*}
F(n,d_1,d_2)= \sum_{a \text{ (mod $n$)}} \varrho_{d_1}(a) g(a-d_2) = \sum_{a+b\equiv d_2 \text{ (mod $n$)}} \varrho_{d_1}(a) g(b)
= (\varrho_{d_1} \times g)(d_2),
\end{equation*}
the Cauchy convolution of the functions $\varrho_{d_1}$ and $g$. Therefore the function $d_2\mapsto F(n,d_1,d_2)$ is also $n$-even (where
$d_1$ is fixed). So, if $(b,n)=d_2$, then $F(n,d_1,b)=F(n,d_1,(b,n))=F(n,d_1,d_2)$.
\end{proof}

\section{Historical remarks} \label{Section_Historical_remarks}

\subsection{Remarks on the paper by Menon} \label{Sect_proof_Menon}

Menon \cite{Men1965} gave three different proofs of \eqref{Menon_id}. His first proof is by using the orbit counting lemma. See
Section \ref{Section_proof_orbit_counting}. The third proof is by using Vaidyana\-thas\-wa\-my's class division of integers (mod $n$). See Section \ref{Section_Vaidya_class_div}. The second proof given by Menon \cite{Men1965} is purely number-theoretic, by using multiplicative functions
of several variables. Namely, he proved the following result. See the survey paper by the author \cite{Tot2014} on multiplicative functions of several
variables. Also see Subbarao \cite{Sub1968}.

\begin{lemma}[{\cite[Lemma]{Men1965}}] \label{Lemma_Menon} Let $f$ be a multiplicative arithmetic function of $r$ variables and let
$P_i\in \Z[x]$ \textup{($1\le i\le r$)} be polynomials. Then the function
\begin{align} \label{function_F}
F(n):= \sum_{a=1}^n f((P_1(a),n),\ldots,(P_r(a),n))
\end{align}
is multiplicative in the single variable $n$, where $(P_i(a),n)$ \textup{($1\le i\le r$)} is the gcd of $P_i(a)$ and $n$.
\end{lemma}

For the proof, he showed that for any $n_1,n_2\in \N$ with $(n_1,n_2)=1$ one has $F(n_1n_2)=F(n_1)F(n_2)$. Then, by applying Lemma
\ref{Lemma_Menon} in the case $r=2$, $P_1(x)=x-1$, $P_2(x)=x$ and $f(n_1,n_2)=f(n_1)E_0(n_2)$, where $E_0(n)=1$ for $n>1$, $E_0(1)=1$,
the following result is deduced.

\begin{theorem}[{\cite[Th.\ 1]{Men1965}}] \label{Th_1_Menon} Let $f$ be a multiplicative arithmetic function of a single variable. Then
\begin{equation} \label{Menon_f}
\sum_{\substack{a=1 \\ (a,n)=1}}^n f((a-1,n))= \prod_{p^\nu\mid \mid n} \left( \sum_{j=0}^\nu \varphi(p^{\nu-j})f(p^j)- 
p^{\nu-1} \right).
\end{equation}
\end{theorem}

Here \eqref{Menon_f} is the same as identity \eqref{product_form_f}, deduced by different arguments, and it has been recovered by Sivaramakrishnan \cite[Th.\ 6.1]{Siv1978i} by the theory of $n$-even functions. 

By selecting $f(n)=n$ in Theorem \ref{Th_1_Menon} we have the original identity \eqref{Menon_id}. Some other particular choices of $f$ are also given. For example, see \cite[p.\ 159]{Men1965},
\begin{equation} \label{Menon_sigma}
\sum_{\substack{a=1 \\ (a,n)=1}}^n \sigma((a-1,n))= n \prod_{p^\nu\mid \mid n} \left(1+\nu-\frac1{p}\right),
\end{equation}
\begin{equation} \label{Menon_varphi}
\sum_{\substack{a=1 \\ (a,n)=1}}^n \varphi((a-1,n))= n \prod_{p^\nu \mid \mid n} \left(1+\nu - \frac{2\nu+1}{p}+ \frac{\nu-1}{p^2} \right).
\end{equation}

The case $f(n)=n^k$ has been treated by Sivaramakrishnan \cite{Siv1974}. It has been pointed out by T\u{a}rn\u{a}uceanu \cite[Cor.\ 2]{Tar2021arXiv}, applying the weighted form of the orbit counting lemma that
\begin{equation} \label{Menon_result_sigma}
\sum_{\substack{a=1 \\ (a,n)=1}}^n g((a-1,n))= \varphi(n)\sigma(n),
\end{equation}
where $g=\id \varphi*\1$, i.e., $g(n)=\sum_{d\mid n} d\varphi(d)$. Here \eqref{Menon_result_sigma} immediately follows from identity \eqref{Menon_f_convo_form}.

Menon also proved the following result, related to the gcd-sum function \eqref{Pillai}.

\begin{theorem}[{\cite[Th.\ 2]{Men1965}}] \label{Th_2_paper_Menon} Let $f$ be a multiplicative arithmetic function of a single variable. Then
\begin{equation} \label{id_a^2-a}
\sum_{a=1}^n f((a^2-a,n))= \prod_{p^\nu\mid \mid n} \left( 2 \sum_{j=0}^\nu \varphi(p^{\nu-j})f(p^j)-p^\nu \right).
\end{equation}
\end{theorem}

Then he presented some of its corollaries. For example, see \cite[p.\ 161]{Men1965},
\begin{equation*} 
\sum_{a=1}^n \sigma((a^2-a,n))= n \tau(n^2).
\end{equation*}

Certain generalizations of \eqref{id_a^2-a} have been given by Venkatramaiah \cite{Ven1975}. 

Note that by the convolution method, similar to that given in Section \ref{Section_Method_IV}, for the function \eqref{function_F} one has
\begin{equation} \label{F_convo}
F(n)= n \sum_{d_1\mid n,\ldots, d_r\mid n} \frac{(\mu_r*_rf)(d_1,\ldots,d_r)}{[d_1,\ldots,d_r]}N(d_1,\ldots,d_r),
\end{equation}
valid for any function $f$ of $r$ variables, where $\mu_r$ is the $r$ variables M\"{o}bius function, defined by
$\mu_r(d_1,\ldots,d_r)=\mu(d_1)\cdots \mu(d_r)$, $*_r$ is the $r$ variables Dirichlet convolution, and
$N(d_1,\ldots,d_r)$ denotes the number of solutions (mod $[d_1,\ldots,d_r]$) of the simultaneous congruences
$P_1(x)\equiv 0$ (mod $d_1$), ..., $P_r(x)\equiv 0$ (mod $d_r$). If $f$ is multiplicative, then the convolution representation
\eqref{F_convo} shows that $F$ is also multiplicative, and its values can be computed for special choices of the polynomials $P_1(x), \ldots,
P_r(x)$. This is remarked by Haukkanen and T\'oth \cite[Remark\ 3.2]{HauTot2020}.

\subsection{Some more remarks}

The orbit counting lemma was also applied by Richards \cite{Ric1984} and Fung \cite{Fun1994} to deduce Menon's identity \eqref{Menon_id}. They do not cite the paper by Menon \cite{Men1965}. Richards \cite{Ric1984} also proved the following result: Let $G$ be a group of order $n$ and let
$c(G)$ denote the number of its cyclic subgroups. Then $c(G)\ge \tau(n)$ with equality if and only if $G$ is cyclic.

It was noted by Sivaramakrishnan \cite{Siv1969} that for every $n\in \N$,
\begin{equation} \label{id_b}
\sum_{a=1}^n b((a,n)) = \varphi(n)\tau(n),
\end{equation}
where $b(n) = \prod_{p^\nu \mid\mid n} (p^\nu-1)$, called Venkataraman's block-function in \cite{Siv1969}. We remark that 
$b(n)$ is identical to the unitary Euler function $\varphi^*(n)$, defined by Cohen \cite{Coh1960} as the number of integers 
$a$ such that $1\le a\le n$ and $(a,n)_*=1$, where $(a,n)_*=\max \{d\in \N: d\mid a, d\mid \mid n\}$.

In fact, if $f$ is an arbitrary function, then 
\begin{equation} \label{Pillai_f}
P_f(n):= \sum_{a=1}^n f((a,n)) =  \sum_{d\mid n} f(d) \varphi(n/d),
\end{equation}
the Dirichlet convolution of the function $f$ and the Euler function $\varphi$, which follows by grouping 
the terms according to the values $(a,n)=d$. If $f=\id$, then this reduces to identity \eqref{Pillai} 
regarding the Pillai function $P$. If $f$ 
is multiplicative, then $P_f$ is multiplicative as well. Now identity \eqref{id_b} is an easy consequence of 
\eqref{Pillai_f}.

Another direct consequence of \eqref{Pillai_f} is the polynomial identity
\begin{equation} \label{id_polynom}
\prod_{a=1}^n \left(x^{(a,n)}-1\right) =  \prod_{d\mid n} \left(x^d-1\right)^{\varphi(n/d)},
\end{equation}
valid for every $n\in \N$, which is formally obtained for $f(n)=\log (x^n-1)$, and may be compared to the well-known
identity
\begin{equation*} 
\Phi_n(x) =  \prod_{d\mid n} \left(x^d-1\right)^{\mu(n/d)}
\end{equation*}
concerning the cyclotomic polynomials $\Phi_n(x)$.

Menon's identity \eqref{Menon_id} and some of its generalizations are included in the following textbooks and monographs:
Freud and Gyarmati \cite[Ex.\ 6.1.16]{FG2020}, McCarthy \cite[Chs.\ 1,2]{McC1986}, Niven, Zuckerman and Montgomery \cite[Ex.\ 4.2.25]{NivZucMon1991}, Sivaramakrishnan \cite[Ch.\ VIII]{Siv1989}, \cite[Ch.\ 7]{Siv2007}.

\section{Generalizations and analogs}

\subsection{Identities for $f((a-1,n))$}

As already mentioned in Sections \ref{Section_Vaidya_class_div} and \ref{Sect_proof_Menon}, Menon \cite{Men1965} 
evaluated the sum
\begin{equation*}
M_f(n):= \sum_{\substack{a=1 \\ (a,n)=1}}^n f((a-1,n)).
\end{equation*}

In the case $f$ is an arbitrary arithmetic function, he deduced identity \eqref{id_by_Vaidya_method}, which may be compared 
to \eqref{Pillai_f}, concerning the gcd-sum function. If $f$ is a multiplicative function, then $M_f$ is also multiplicative, and Menon \cite{Men1965} obtained identity \eqref{Menon_f}.    

For an arbitrary function $f$, identity \eqref{Menon_f_convo_form} was proved by Sita Ramaiah \cite[Th. 9.1]{Sit1978} in a more general form, concerning regular systems of divisors (see Section \ref{Sect_Regular_convo}) and $k$-reduced residue systems (see Section \ref{Sect_general_gcd}).
For the special choices of $f=\sigma$ and $f=\varphi$ see identities \eqref{Menon_sigma} and \eqref{Menon_varphi}, respectively. Several other functions can be discussed. For example, if $f(n)=\omega(n)=\sum_{p\mid n} 1$, then one has
\begin{equation*}
M_{\omega}(n):= \sum_{\substack{a=1 \\ (a,n)=1}}^n \omega((a-1,n))= \varphi(n)  \sum_{p\mid n} \frac1{p-1},
\end{equation*}
and if $f(n)=\log (x^n-1)$, then one obtains
\begin{equation} \label{id_Menon_cyclot}
\prod_{\substack{a=1\\(a,n)=1}}^n \left(x^{(a-1,n)}-1\right) =  \prod_{d\mid n} \Phi_d(x)^{\varphi(n)/\varphi(d)},
\end{equation}
being a counterpart of identity \eqref{id_polynom}. We are not aware of any references for \eqref{id_polynom} and \eqref{id_Menon_cyclot}.

A slightly different type of identity is the following. Let $k,m_1,\ldots,m_k\in \N$, $M=[m_1,\ldots,m_k]$. T\'oth \cite{Tot2011} proved by number-theoretic 
arguments, as a special case of more general identities that
\begin{equation} \label{Menon_Toth}
\frac1{\varphi(M)} \sum_{\substack{a \text{ (mod $M$)}\\(a,M)=1}} (a-1,m_1)\cdots (a-1,m_k) = \sum_{d\mid m_1,\ldots, d_k\mid m_k} \frac{\varphi(d_1)\cdots \varphi(d_k)}{\varphi([d_1,\ldots,d_k])}.
\end{equation}

If $k=1$, then \eqref{Menon_Toth} reduces to Menon's identity \eqref{Menon_id}. Note that, according to the orbit counting lemma, 
\begin{equation*} 
\frac1{\varphi(q)} \sum_{\substack{a \text{ (mod $q$)}\\(a,q)=1}} (a-1,m_1)\cdots (a-1,m_k),
\end{equation*}
where $q=m_1\cdots m_k$, represents the number of cyclic subgroups of the group $\Cik_{m_1}\times \cdots 
\times \Cik_{m_k}$.

\subsection{Identities concerning $n$-even functions $f$}

The proof of the next identity only uses the definition of $n$-even functions and no further properties of them. 
It is given by T\'oth \cite[Th.\ 2.1]{Tot2018} in a more general form. 

\begin{theorem} \label{Th_even} 
Let $n\in \N$, $s\in \Z$ and let $f$ be an even function \textup{(mod $n$)}. Then
\begin{equation} \label{Menon_even}
\sum_{\substack{a \textup{ (mod $n$)}\\ (a,n)=1}} f(a-s) =  \varphi(n) \sum_{\substack{d\mid n\\ (d,s)=1}} \frac{(\mu*f)(d)}{\varphi(d)}. 
\end{equation}
\end{theorem}

\begin{proof}[Proof of Theorem {\rm \ref{Th_even}}]
Since $f$ is an even function (mod $n$), for every $k\in \N$ we have
\begin{equation*} 
f(k) = f((k,n)) = \sum_{d\mid (k,n)} (\mu*f)(d),
\end{equation*}
and we follow Method IV:
\begin{equation*}
\sum_{\substack{a \text{ (mod $n$)}\\ (a,n)=1}} f(a-s) = 
\sum_{\substack{a=1\\ (a,n)=1}}^n \sum_{d\mid (a-s,n)} (\mu*f)(d) =
\sum_{\substack{a=1\\ (a,n)=1}}^n \sum_{\substack{d\mid a-s\\ d\mid n}} (\mu*f)(d)
\end{equation*}
\begin{equation*}
= \sum_{d\mid n} (\mu*f)(d) \sum_{\substack{a=1\\  (a,n)=1\\ a\equiv s\, \text{\rm (mod $d$)}}}^n 1.
\end{equation*}

By Lemma \ref{Lemma_n_d} the inner sum is $\varphi(n)/\varphi(d)$ if $(d,s)=1$ and $0$ otherwise, which proves identity \eqref{Menon_even}.
\end{proof}

If $f(a)=(a,n)$, then \eqref{Menon_even} gives
\begin{equation} \label{Menon_id_s}
\sum_{\substack{a \text{ (mod $n$)}\\ (a,n)=1}} (a-s,n) =  \varphi(n) \sum_{\substack{d\mid n\\ (d,s)=1}} 1, 
\end{equation}
which reduces to Menon's identity \eqref{Menon_id} for $s=1$. If $f(a)=c_n(a)$, the Ramanujan sum, then 
\eqref{Menon_even} gives the first identity of the formulas
\begin{equation} \label{Ramanujan_id_3_terms}
\sum_{\substack{a \text{ (mod $n$)} \\ (a,n)=1}} c_n(a-s) = \varphi(n) \sum_{\substack{d\mid n \\ (d,s)=1}}
\frac{d\mu(n/d)}{\varphi(d)} = \mu(n)c_n(s),
\end{equation}
the second one being the Brauer-Rademacher identity. See \cite[Ch.\ 2]{McC1986}. Compare \eqref{Ramanujan_id_3_terms} to \eqref{Ramanujan_sum_id}.

\subsection{Geometric and harmonic means of $(a-1,n)$ with $(a,n)=1$}

Menon's identity shows that $\tau(n)$ is the arithmetic mean of the values $(a-1,n)$, where $a$ runs over a reduced residue system (mod $n$).
The following identities are concerning the geometric and harmonic means of these values, denoted by $G(n)$ and $H(n)$, respectively. One has, for any $n\in \N$,
\begin{equation} \label{G}
G(n):= \biggl(\prod_{\substack{a=1\\ (a,n)=1}}^n (a-1,n)\biggr)^{1/\varphi(n)} =\prod_{p\in \P} p^{e_p(n)},
\end{equation}
where
\begin{equation*}
e_p(n) = \frac{p}{(p-1)^2}\left(1-\frac1{p^{\nu_p(n)}} \right)
\end{equation*}
and
\begin{equation} \label{H}
H(n):= \varphi(n) \biggl(\sum_{\substack{a=1\\ (a,n)=1}}^n \frac1{(a-1,n)} \biggr)^{-1} =
\prod_{p \in \P} \left(1-\frac{p}{p^2-1} \left( 1-\frac1{p^{2\nu_p(n)}} \right) \right)^{-1},
\end{equation}
using the notation $n=\prod_{p\in \P} p^{\nu_p(n)}$.

Identity \eqref{G} was obtained, in a more general form by Sita Ramaiah \cite[Cor.\ 9.3]{Sit1978}.
For the proof choose $f(n)=\log n$ in \eqref{Menon_f_convo_form} and use that $\mu*\log=\Lambda$, the von Mangoldt function. This gives
\begin{equation*}
\log G(n)= \sum_{\substack{a=1\\ (a,n)=1}}^n  \log \, (a-1,n) = \varphi(n) \sum_{d\mid n} \frac{\Lambda(d)}{\varphi(d)},
\end{equation*}
which is not multiplicative, but can be computed, leading to formula \eqref{G}.

To prove \eqref{H} choose $f(n)=1/n$ in \eqref{Menon_f_convo_form} and note that $(\mu*f)(p^\nu)=(1-p)/p^\nu$ for every prime power $p^\nu$ ($\nu \ge 1$).

\subsection{Identities using a polynomial $P$}

Richards \cite{Ric1984} stated without proof that for any polynomial $P$ with integer coefficients one has
\begin{equation} \label{Richards_id}
\sum_{\substack{a=1\\ (a,n)=1}}^n (P(a),n)=\varphi(n) \sum_{d\mid n} N_P(d),
\end{equation}
where $N_P(d)$ is the number of incongruent solutions $x$ of the congruence $P(x)\equiv 0$ (mod $d$) such that $(x,d)=1$.

A short proof is the following. By the Gauss formula $\sum_{d\mid n}\varphi(d)=n$ the left hand side of \eqref{Richards_id}, say $R$, is
\begin{equation*}
R= \sum_{\substack{a=1\\ (a,n)=1}}^n \sum_{d\mid (P(a),n)} \varphi(d) = \sum_{d\mid n} \varphi(d)
\sum_{\substack{a=1\\ (a,n)=1\\ P(a)\equiv 0 \text{ (mod $d$)}}}^n 1.
\end{equation*}

According to
Lemma \ref{Lemma_n_d}, to each such solution $x$ one correspond $\varphi(n)/\varphi(d)$ solutions $y$ (mod $n$) such that $(y,n)=1$. Hence
\begin{equation*}
R= \sum_{d\mid n} \varphi(d) \frac{\varphi(n)}{\varphi(d)}N_P(d)= \varphi(n) \sum_{d\mid n} N_P(d).
\end{equation*}

If $f(x)=x-s$ with $s\in \Z$, then \eqref{Richards_id} implies identity \eqref{Menon_id_s}.

In fact, we have the following more general result of which proof is similar.

\begin{theorem} If $f$ is an arbitrary function and $P(x)$ is any polynomial with integer coefficients, then
\begin{equation} \label{Richards_id_f}
M_{f,P}(n):=\sum_{\substack{a=1\\ (a,n)=1}}^n f(P(a),n)) = \varphi(n) \sum_{d\mid n} \frac{(\mu*f)(d)}{\varphi(d)} N_P(d).
\end{equation}
\end{theorem}

If $f$ is multiplicative, then $M_{f,P}$ is also multiplicative. The case of the polynomial $P(x)=x^j$ ($j\in \N$) and general identities concerning systems of polynomials have been discussed by T\'oth \cite{Tot2011}. 
 
\subsection{Identities concerning parameters $m,n,t$}

It was proved by Sivaramakrishnan \cite{Siv1969}, in a slightly different form, that
\begin{equation} \label{Menon_id_Siv}
M(m,n,t):= \sum_{\substack{a=1\\ (a,m)=1}}^t (a-1,n)=\frac{t\, \varphi(m)\tau(n)}{m} \prod_{p^\nu\mid \mid n_1}
\left(1-\frac{\nu}{(\nu+1)p} \right),
\end{equation}
where $m,n,t\in \N$ such that $m\mid t$, $n\mid t$ and $n_1=\max \{d\in \N: d\mid n, (d,m)=1\}$. If $m=n=t$, then
$M(n,n,n)=M(n)$, that is, \eqref{Menon_id_Siv} reduces to \eqref{Menon_id}. Sivaramakrishnan \cite{Siv1969} pointed 
out that if $t=[m,n]$, the least common multiple of $m$ and $n$, then $M(m,n,[m,n])$ is a multiplicative function of two variables. However, if $n\mid m$ and $t=m$, then it
follows from \eqref{Menon_id_Siv} that
\begin{equation} \label{Menon_n_divides_m}
\sum_{\substack{a=1\\ (a,m)=1}}^m (a-1,n) = \varphi(m) \tau(n).
\end{equation}

Identity \eqref{Menon_n_divides_m} was recently obtained by Jafari and Madadi \cite[Cor.\ 2.2]{JafMad2017}, as a corollary
of the following group theoretic result, without referring to the paper \cite{Siv1969}, by using the orbit counting lemma. Let $G$ be a finite group of order $m$ and let $n$ be a
divisor of $m$. Let $A_G(n)$ denote the number of
cyclic subgroups of $G$ of which orders divide $n$ and let $B_G(n)$ denote the number of solutions in $G$ of the equation $x^n = 1$. Then, see
\cite[Th.\ 2.1]{JafMad2017},
\begin{equation*}
A_G(n)= \frac1{\varphi(m)} \sum_{\substack{a=1\\ (a,m)=1}}^m B_G((a-1,n)).
\end{equation*}

A short direct proof of \eqref{Menon_n_divides_m} can be obtained, e.g., by Method IV:
\begin{equation*}
\sum_{\substack{a=1\\ (a,m)=1}}^m (a-1,n) = \sum_{\substack{a=1\\ (a,m)=1}}^m 
\sum_{d\mid (a-1,n)} \varphi(d) =
\sum_{d\mid n} \varphi(d) \sum_{\substack{a=1\\  (a,m)=1\\ a\equiv 1 \text{\rm (mod $d$)}}}^m 1.
\end{equation*}

Here $d\mid n$ and $n\mid m$ imply that $d\mid m$, and the inner sum is $\varphi(m)/\varphi(d)$ by Lemma \ref{Lemma_n_d}, with $m$ instead of $n$.

\subsection{Sums over $k$-reduced residue systems (mod $n^k$)} \label{Sect_general_gcd}

Let $k\in \N$ and let $(a,b)_k$ denote the greatest common $k$-th power divisor of $a,b\in \N$. A $k$-reduced residue system (mod $n^k$) is defined as the subset of elements $a$ of a complete residue system (mod $n^k$) such that $(a,n^k)_k=1$. If $k=1$, then this is the usual notion of a reduced residue system (mod $n$). 

The number of elements in a $k$-reduced residue system (mod $n^k$) is denoted by $\phi_k(n)$, function introduced by Cohen \cite{Coh1949}. Here $\phi_1(n)=\varphi(n)$. Also see Cohen \cite{Coh1958}.

Note that $\phi_k(n)=J_k(n)$, the Jordan function, for every $n,k\in \N$. A direct proof is the following: Consider the set
$A=\{a\in \N: 1\le a\le n^k \}$. Group the elements $a$ of $A$ according to the values $(a,n^k)_k=d^k$, where $d\mid n$.
Here $a=d^k b$, with $1\le b\le (n/d)^k$ and $(b,(n/d)^k)_k=1$. By the definition of the function $\phi_k$, the number of such elements $a$ is
exactly $\phi_k(n/d)$. Therefore, $n^k=\sum_{d\mid n} \phi_k(n/d)$. By M\"{o}bius inversion we deduce that $\phi_k(n)=\sum_{d\mid n} d^k \mu(n/d)=J_k(n)$.

Nageswara Rao \cite[Th.\ 2]{Nag1972} proved by using $n$-even functions that
\begin{equation*} 
\sum_{\substack{a \text{ (mod $n^k$)} \\ (a,n^k)_k=1}} (a-s,n^k)_k = \phi_k(n) \tau(n),
\end{equation*}
where $s\in \Z$, $(s,n^k)_k=1$. For $k=1$ one has Menon's identity.

\subsection{Identities with $(a,n)\in S$, where $S$ is a subset of $\N$}

Let $S$ be an arbitrary nonempty subset of $\N$. Define the function $\mu_S$ by $\mu_S*\1 =\varrho_S$, where
$\varrho_S$ is the characteristic function of $S$. That is,
\begin{equation*}
\sum_{d\mid n} \mu_S(n) = \begin{cases} 1, & \text{ if $n\in S$},\\ 0, & \text{ otherwise}.
\end{cases}    
\end{equation*}

Furthermore, let $\varphi_S(n)$ denote the number of integers $a$ (mod $n$) such that $(a,n)\in S$. It can be shown, 
see Cohen \cite[Cor.\ 4.1]{Coh1959}, that $\varphi_S=\mu_S*\id$, that is,
\begin{equation*}
\varphi_S(n)= \sum_{d\mid n} d \mu_S(n/d) \quad (n\in \N).    
\end{equation*}

If $S=\{1\}$, then $\mu_{\{1\}}=\mu$ and $\varphi_{\{1\}}=\varphi$ are the classical M\"obius function and Euler function, respectively. If $S$ is the set of squares, then $\mu_S=\lambda$ is the Liouville function and $\phi_S= \beta$, given by
\begin{equation*}
\beta(n)= \sum_{d\mid n} d\, \lambda(n/d),
\end{equation*}
which is identical to the alternating sum-of-divisors function. See the author \cite{Tot2013Sap}. Several other special choices of $S$ can be investigated. To mention a further one, if $S$ is the set of $k$-th power free integers, then $\varphi_S$ is the Klee function. See Klee \cite{Kle1948}.

If $S$ is an arbitrary subset, $n\in \N$ and $s\in \Z$ then
\begin{equation} \label{Menon_S_arbitrary}
\sum_{\substack{a \text{ (mod $n$)} \\ (a,n)\in S}} (a-1,n) = (\id * (\mu_S \times \varphi))(n),
\end{equation}
and 
\begin{equation} \label{Ramanujan_id_S}
\sum_{\substack{a \text{ (mod $n$)} \\ (a,n)\in S}} c(a-s,n) = \mu_S(n) c_n(s),
\end{equation}
which reduce to Menon's identity \eqref{Menon_id}, respectively \eqref{Ramanujan_sum_id} if $S=\{ 1\}$. 

A subset $S$ of $\N$ is called multiplicative if $1\in S$ and the characteristic function $\varrho_S$ is multiplicative. In this case
\begin{equation} \label{Menon_S_multipl}
\sum_{\substack{a \text{ (mod $n$)} \\ (a,n)\in S}} (a-1,n) = \prod_{p^\nu \mid\mid n} \left( P(p^\nu) +\varphi_S(p^\nu)-p^\nu \right),
\end{equation}
where $P$ is the gcd-sum function.

In particular, if $S$ is the set of squares, then we have
\begin{equation} \label{Menon_squares}
\sum_{\substack{a \text{ (mod $n$)} \\ (a,n) \text{ is a square}}} (a-1,n) = \prod_{p^\nu \mid\mid n} \left( P(p^\nu) - \beta(p^{\nu-1}) \right)
\end{equation}
and 
\begin{equation} \label{Ramanujan_squares}
\sum_{\substack{a \text{ (mod $n$)} \\ (a,n) \text{ is a square}}} c(a-s,n) = \lambda(n) c_n(s).
\end{equation}

Identities \eqref{Ramanujan_id_S} and \eqref{Menon_S_multipl} have been deduced by Haukkanen \cite[Th.\ 3.1, 3.2]{Hau1997} by finite Fourier-expansions. They can be proved by different arguments as well, e.g., by Method V, which also leads to the general identity \eqref{Menon_S_arbitrary}, to be compared to \eqref{Menon_convo}. 
In the case when $S$ is the set of squares, identity \eqref{Menon_squares} was obtained by
Sivaramakrishnan \cite[Th.\ 6.2]{Siv1979}, and identity \eqref{Ramanujan_squares} by Haukkanen and Sivaramakrishnan
\cite[Sect.\ 4]{HauSiv1991}.

A Menon-type identity related to Klee's arithmetic function was discussed by Chandran, Thomas, and Namboothiri \cite{CTNCzech}.

\subsection{Unitary analogs}

Let $(a,b)_*$ denote the greatest divisor of $a$ which is a unitary divisor of $b$.
Nageswara Rao \cite{Nag1966} proved via the ,,unitary'' version of Method X that
\begin{equation} \label{Menon_id_unit}
 \sum_{\substack{a \text{ (mod $n$)} \\ (a,n)_*=1}} (a-1,n)_* = \varphi^*(n) \tau^*(n),
\end{equation}
where $\varphi^*(n)=\prod_{p^\nu \mid \mid n} (p^\nu-1)$ is the unitary Euler function and $\tau^*(n)=\sum_{d\mid \mid n} 1$ is the unitary divisor function.
Note that $\tau^*(n)=2^{\omega(n)}$, where $\omega(n)=\sum_{p\mid n} 1$.

Short direct proofs can be given by Methods I-V as well, using properties of unitary functions. 
Some related identities are 
\begin{equation} \label{Menon_unit_usual}
 \sum_{\substack{a \text{ (mod $n$)} \\ (a,n)_*=1}} (a-1,n) = \prod_{p^\nu \mid\mid n} \left((\nu+1)p^\nu-\nu p^{\nu-1}-1 \right),
\end{equation}
\begin{equation} \label{Menon_usual_unit}
 \sum_{\substack{a \text{ (mod $n$)} \\ (a,n)=1}} (a-1,n)_* = \prod_{p^\nu \mid\mid n} \left(2p^\nu-p^{\nu-1}-1 \right),
\end{equation}
where \eqref{Menon_unit_usual} is a special case of identity \eqref{Menon_Hau}, and we are not aware of any reference for \eqref{Menon_usual_unit}.

\subsection{Regular systems of divisors} \label{Sect_Regular_convo}

Let $A(n)$ be a subset of the set of positive divisors of $n$ for every $n\in \N$. 
The $A$-convolution of the functions
$f, g : \N \to \C$ is defined by
\begin{equation*}
(f *_A g)(n) = \sum_{d\in A(n)} f(d)g(n/d) \quad (n\in  \N).
\end{equation*}

Let ${\cal A}$ denote the set of arithmetic functions $f:\N \to \C$. The system $A = (A(n))_{n\in \N}$ of divisors is called regular, cf. Narkiewicz \cite{Nar1963}, if the following conditions hold true:

(a) $({\cal A},+, *_A)$ is a commutative ring with unity,

(b) the $A$-convolution of multiplicative functions is multiplicative,

(c) the constant $1$ function has an inverse $\mu_A$ (generalized M\"{o}bius function)
with respect to $*_A$ and $\mu_A(p^\nu)\in  \{-1, 0\}$ for every prime power $p^\nu$ ($\nu \ge  1$).

It can be shown that the system $A = (A(n))_{n\in \N}$ is regular if and only if

(i) $A(mn) = \{de : d \in A(m), e\in  A(n)\}$ for every $m, n\in \N$ with $(m, n) = 1$,

(ii) for every prime power $p^\nu$ ($\nu \ge 1$) there exists a divisor $t = t_A(p^\nu)$ of $\nu$, called
the type of $p^\nu$ with respect to $A$, such that
\begin{equation*}
A(p^{it}) = \{1, p^t, p^{2t},\ldots , p^{it}\}
\end{equation*}
for every $i\in \{0, 1,\ldots, \nu/t\}$.

Examples of regular systems of divisors are $A = D$, where $D(n)$ is the set
of all positive divisors of $n$ and $A = U$, where $U(n)$ is the set of unitary divisors of $n$.
For every prime power $p^\nu$ ($\nu\ge 1$) one has $t_D(p^\nu) = 1$ and $t_U(p^\nu) = \nu$. Here $*_D$ and $*_U$ are the Dirichlet convolution and
the unitary convolution, respectively. 

For properties of regular convolutions and related arithmetical functions we refer to Narkiewicz \cite{Nar1963}, McCarthy \cite{McC1968,McC1986} and Sita Ramaiah \cite{Sit1978}.

Let $\tau_A(n)=\sum_{d\in A(n)}1 $ denote the corresponding divisor function and let
\begin{equation*}
\varphi_A(n) = \sum_{\substack{a \text{ (mod n)} \\ (a,n)_A=1}} 1,
\end{equation*}
be the generalized Euler function, where
\begin{equation*}
(a,n)_A = \max \{d\in \N: d\mid a, d\in A(n)\}.
\end{equation*}

If $A$ is a regular system of divisors, then one has
\begin{equation*} 
 \sum_{\substack{a \text{ (mod $n$)} \\ (a,n)_A=1}} (a-1,n)_A = \varphi_A(n) \tau_A(n),
\end{equation*}
proved by Sita Ramaiah \cite{Sit1978}, in a more general form, which recovers identities \eqref{Menon_id} 
and \eqref{Menon_id_unit} in the cases $A=D$ and $A=U$, respectively.  Further generalizations with respect to $n$-even functions of several variables were given by Haukkanen and McCarthy \cite{HauMcC1991}. 

\subsection{A generalized divisibility relation}

Haukkanen \cite{Hau1999} defined a generalized divisibility relation on $\N$ as follows. For every prime $p$ let $f_p:\N \to \N\cup \{0\}$ be a 
function and let $f = (f_p)_{p\in \P}$ be the system of these functions. We say that $\wr$ is a divisibility relation of type $f$ if

1) $1 \, \wr \, n$ for all $n\in \N$,

2) for every prime $p$ and $\nu\in \N \cup \{0\}$, if $d \, \wr \, p^\nu$, then $d = p^i$ for some $i \in \{0, 1, \ldots, \nu\}$,

3) $d \, \wr \, n$ if and only if $p^{\nu_p(d)} \, \wr \, p^{\nu_p(n)}$ for all primes $p$, where $d =\prod_{p\in \P} p^{\nu_p(d)}$ and
$n = \prod_{p\in \P} p^{\nu_p(n)}$ are the canonical factorizations of $d$ and $n$,

4) for every prime $p$ and $\nu \in \N$,  $f_p(\nu)$ is the smallest integer $i \in \{1, 2, \ldots, \nu \}$
such that $p^i \, \wr \, p^\nu$ if such $i$ exists, and $f_p(\nu) = 0$ otherwise.

It is clear that if $d \, \wr \, n$, then $d \mid n$. The usual divisibility relation $\mid$
is a divisibility relation of type $(\1, \1, \ldots)$, where $\1(\nu)  = 1$ for all $\nu \in \N$ and all $p\in \P$.
The unitary divisibility relation $\mid \mid$
is a divisibility relation of type $(\id, \id, \ldots)$, where $\id(\nu)  = \nu$ for all $\nu \in \N$ and all $p\in \P$.
More generally, if $A$ is a regular system of divisors (see Section \ref{Sect_Regular_convo}), then 
the $A$-divisibility relation $\mid_A$, where $d\mid_A n$ if and only if $d\in A(n)$ is a divisibility relation of type
$f=(f_p)$ with $f_p(\nu)= t_A(p^\nu)$, the type of $p^\nu$. Some more examples were also given in \cite{Hau1999}. 

Let $\wr$ be a divisibility relation of type $f$. We denote $(m, n)_{\wr} = \max\{d : d\mid m, d\, \wr \, n\}$. The generalized Euler function
$\phi_{\wr}(n)$ is defined as the number of integers $a$ (mod $n$) such that $(a,n)_{\wr}=1$.

Haukkanen \cite[Cor.\ 4.1]{Hau2005} proved that  for every $n\in \N$,
\begin{equation} \label{Menon_Hau}
 \sum_{\substack{a \text{ (mod $n$)} \\ (a,n)_{\wr}=1}} (a-1,n) = \phi_{\wr}(n) \sum_{d\mid n} \frac{\varphi(d)n_d}{d\phi_{\wr}(n_d)},
\end{equation}
where $\varphi$ is the classical Euler's function and $n_d=\prod_{p^\nu \mid \mid n, \, p\mid d} p^\nu$.

Schwab and Bede \cite{SchBed2014} considered the ”odd-divisibility” relation $\mid_{\otimes}$
defined by $d \mid_{\otimes} n $ if and only if $d$ is odd and $d\mid n$. They refer to the paper by 
Haukkanen \cite{Hau2005}, remark that this is a divisibility relation of type $(\theta,\1,\1,\ldots)$, where
$\theta(\nu)=0$ ($\nu \in \N$), and as an application of \eqref{Menon_Hau} deduce that
for every $n\in \N$,
\begin{equation*}
\sum_{\substack{a \text{ (mod $n$)} \\ (a,n)_{\otimes}=1}} (a-1,n) = \varphi_{\otimes}(n) \left(\tau(n)-\frac1{2}\tau_2(n)\right),
\end{equation*}
where $(a,n)_{\otimes}$ denotes the greatest common odd divisor of $a$ and $n$, $\varphi_{\otimes}(n)$ is the 
number of integers $a$ (mod n) such that $(a,n)_{\otimes}=1$, $\tau(n)$ is the number of divisors of $n$, 
and $\tau_2(n)$ is the number of even divisors of $n$.

\subsection{Regular integers (mod $n$)}

An integer $a$ is called regular (mod $n$), if there exists an integer $x$ such that $a^2x\equiv a$ (mod $n$), i.e., the residue class of $a$
is a regular element (in the sense of J. von Neumann) of the ring $\Z_n$ of residue classes (mod $n$).

Let $n>1$ be an integer with prime factorization $n=p_1^{\nu_1}\cdots p_r^{\nu_r}$. It can be shown that the following three statements are equivalent: (i) $a\in \N$ is regular (mod $n$), (ii) $(a,n)\mid \mid n$,
(iii) for every $i\in \{1,\ldots,r\}$ either $p_i\nmid a$ or $p_i^{\nu_i}\mid a$.

Let $\Reg_n=\{a: 1\le a\le n$, $a$ is regular (mod $n$)$\}$ and let $\varrho(n)=\# \Reg_n$ denote the number of regular integers $a$ (mod $n$) such that $1\le a\le n$. One has $\varrho(n)=\sum_{d\mid\mid n} \varphi(d)$. See T\'oth \cite{Tot2008} for more properties.

Apostol and T\'oth \cite[Prop.\ 5.3.1]{ApoTot2015} proved that
\begin{equation*}
\sum_{a \in  \Reg_n} (a-1,n)= \sum_{d\mid \mid n} \varphi(d)\tau(d) = \prod_{p^\nu \mid \mid n} 
\left(p^{\nu-1}(p-1)(\nu+1)+1 \right).
\end{equation*}

\subsection{Multiple sums}

Nageswara Rao \cite[Th.\ 1]{Nag1972} proved by the theory of $n$-even functions that
\begin{equation} \label{Nageswara_id}
\sum_{\substack{a_1,\ldots,a_k=1\\ (a_1,\ldots, a_k,n)=1}}^n
(a_1-s_1,\ldots,a_k-s_k,n)^k = J_k(n) \tau(n) \quad (k,n\in
\N),
\end{equation}
where $s_1,\ldots,s_k\in \Z$, $(s_1,\ldots,s_k,n)=1$ and
$J_k(n)=n^k\prod_{p\mid n} (1-1/p^k)$ is the Jordan function of
order $k$. 

Haukkanen and Sivaramakrishnan \cite[Th.\ 9]{HauSiv1994} deduced, also by using $n$-even functions that  
\begin{equation}  \label{HauSiv_id}
\sum_{\substack{a_1,\ldots,a_k=1\\ (a_1,\ldots, a_k,n)=1}}^n
(a_1-1,\ldots,a_k-1,n) = J_k(n) \sum_{d\mid n} \frac1{\varepsilon_k(d)} \quad (k,n\in
\N),
\end{equation}
where $\varepsilon_k(d)=J_k(d)/\varphi(d)$ is the number of cyclic subgroups of order $d$ of the group $\Z_d \times \cdots \times \Z_d$ with $k$ factors.

Sury \cite{Sur2009} used the orbit counting lemma to prove the following identity:
\begin{equation} \label{Sury_id}
\sum_{\substack{a_1,a_2,\ldots,a_k=1\\ \gcd(a_1,n)=1}}^n (a_1-1,a_2,\ldots,a_k,n)= \varphi(n) \sigma_{k-1}(n) \quad (k,n\in \N),
\end{equation}
where $\sigma_s(n)=\sum_{d\mid n} d^s$. 

For $k=1$ all of \eqref{Nageswara_id}, \eqref{HauSiv_id} and \eqref{Sury_id} recover Menon's identity.

Direct proofs of these and other similar identities (see, e.g., \cite{Ven1973}) can be given by Method IV, for example. 
To deduce \eqref{Nageswara_id} write
\begin{equation*}
T:= \sum_{\substack{a_1,\ldots,a_k=1\\ (a_1,\ldots, a_k,n)=1}}^n (a_1-s_1,\ldots,a_k-s_k,n)^k 
\end{equation*}
\begin{equation*}
=\sum_{\substack{a_1,\ldots,a_k=1\\ (a_1,\ldots, a_k,n)=1}}^n \sum_{d\mid (a_1-s_1,\ldots,a_k-s_k,n)} J_k(d) =
\end{equation*}
\begin{equation*}
= \sum_{d\mid n} J_k(d) \sum_{\substack{a_j\equiv s_j \text{ (mod $d$)}\\ 1\le j\le k }}
\sum_{\delta \mid (a_1,\ldots,a_k,n)=1} \mu(\delta)
\end{equation*}\(\)
\begin{equation*}
= \sum_{d\mid n} J_k(d) \sum_{\delta \mid n} \mu(\delta) \sum_{\substack{1\le b_j\le n/\delta\\ \delta b_j\equiv s_j \text{ (mod $d$)}\\
1\le j\le k }} 1,
\end{equation*}
where the linear congruences $\delta b_j\equiv s_j$ (mod $d$) ($1\le j\le k$) have solutions $b_j$ if and only if $(d,\delta)\mid s_j$ ($1\le j\le k$),
equivalent to $(d,\delta)=1$, since $(d,\delta)\mid n$ and $(s_1,\ldots,s_k,n)=1$ by the given condition. If $(d,\delta)=1$ is satisfied, then
each congruence has one solution (mod $d$) and $n/(d\delta)$ solutions (mod $n/\delta)$. Hence
\begin{equation*}
T= \sum_{d\mid n} J_k(d) \sum_{\substack{\delta \mid n\\ (d,\delta)=1}} \mu(\delta) \left(\frac{n}{d\delta} \right)^k
= n^k \sum_{d\mid n} \frac{J_k(d)}{d^k} \sum_{\substack{\delta \mid n\\ (d,\delta)=1}} \frac{\mu(\delta)}{\delta^k}
\end{equation*}
\begin{equation*}
= n^k \sum_{d\mid n} \frac{J_k(d)}{d^k} \frac{J_k(n)/n^k}{J_k(d)/d^k}= J_k(n) \tau(n),
\end{equation*}
completing the proof of \eqref{Nageswara_id}.

T\'oth \cite[Prop.\ 2.7]{Tot2013} proved by arithmetic arguments that for every $n,k\in \N$,
\begin{equation*}
\sum_{\substack{a_1,\ldots, a_k=1\\ (a_1\cdots a_k,n)=1}}^n (a_1\cdots a_k-1,n)= \varphi(n)^k \tau(n),
\end{equation*}
which reduces to Menon's original identity for $k=1$.

Another multidimensional generalization of Euler's arithmetic function and of Menon's identity has been obtained by T\'oth \cite{Tot2020}. For $k\in \N$ let define the function $\varphi_k(n)$ as
\begin{equation*} 
\varphi_k(n):= \sum_{\substack{a_1,\ldots, a_k=1 \\ (a_1\cdots a_k,n)=1\\ (a_1+\cdots+a_k,n)=1}}^n 1,
\end{equation*}
where $\varphi_1(n)=\varphi(n)$ is Euler's function.

Let $f$ be an arbitrary arithmetic function. Then for every $k,n\in \N$,
\begin{equation} \label{Menon_gen_f}
\sum_{\substack{a_1,\ldots, a_k=1 \\ (a_1\cdots a_k,n)=1 \\ (a_1+\cdots+a_k,n)=1}}^n f((a_1+\cdots +a_k-1,n)) = \varphi_k(n)\sum_{d\mid n} \frac{(\mu \ast f)(d)}{\varphi(d)}.
\end{equation}

If $f=\id$, then 
\begin{equation} \label{Menon_gen_k}
\sum_{\substack{a_1,\ldots, a_k=1 \\ (a_1\cdots a_k,n)=1 \\ (a_1+\cdots+a_k,n)=1}}^n (a_1+\cdots +a_k-1,n) = \varphi_k(n)\tau(n).
\end{equation}

In the case $k=2$ \eqref{Menon_gen_k} has been proved by Sita Ramaiah \cite[Cor.\ 10.4]{Sit1978} in the 
general setting of regular convolutions and $k$-reduced residue systems.

\subsection{Identities for arithmetic functions of several variables}

Let $f$ be an arithmetic function of $k$ variables. Haukkanen and T\'oth \cite{HauTot2020} considered
(in a more general form) sums of type 
\begin{equation*}
\sum_{\substack{1\le a_1\le n_1\\ (a_1,n_1)=1}} \cdots \sum_{\substack{1\le a_k\le n_k\\ (a_k,n_k)=1}}
f((a_1-1,n_1), \ldots,(a_k-1,n_k)).
\end{equation*}

In fact, they deduced common generalizations of the Menon-type identities by
Sivaramakrishnan \cite{Siv1969} and Li, Kim, Qiao \cite{LKQ2019}. In particular, an explicit formula was proved in the case $f(n_1,\ldots,n_k)=[n_1,\ldots,n_k]$, the lcm function. If $n_1=\cdots =n_k=n$, then it gives (see \cite[Cor.\ 4.2]{HauTot2020}): For any $n,k\in \N$,
\begin{equation*}
\sum_{\substack{1\le a_1\le n\\ (a_1,n)=1}} \cdots \sum_{\substack{1\le a_k\le n\\ (a_k,n)=1}}
[(a_1-1,n), \ldots,(a_k-1,n)]
\end{equation*}
\begin{equation*}
=  \varphi(n)^k \prod_{p^\nu \mid \mid  n }
\left( 1+k\nu + \sum_{j=1}^{k-1} (-1)^j \binom{k}{j+1} \frac{p^j}{(p-1)^j(p^j-1)} \left(1-\frac1{p^{\nu j}}\right) \right).
\end{equation*}

If $k=1$ we recover Menon's identity \eqref{Menon_id}.

See the paper by Haukkanen and Wang \cite{HW1996} for a general identity with respect to a set of polynomials.

\subsection{Identities concerning Dirichlet characters} \label{Sect_id_Dirichlet_char}

Zhao and Cao \cite{ZhaCao} proved that
\begin{equation} \label{Menon_id_char}
\sum_{a=1}^n (a-1,n) \chi(a)= \varphi(n)\tau(n/d),
\end{equation}
where $\chi$ is a Dirichlet character (mod $n$) with conductor $d$ ($n\in \N$, $d\mid n$). 
If $\chi$ is the principal character (mod $n$), that is
$d=1$, then \eqref{Menon_id_char} reduces to Menon's identity \eqref{Menon_id}.
If $\chi$ is a primitive character (mod $n$), i.e., $d=n$, then \eqref{Menon_id_char} gives
\begin{equation} \label{Menon_id_char_primit}
\sum_{a=1}^n (a-1,n) \chi(a)= \varphi(n).
\end{equation}

Note that the sums in \eqref{Menon_id_char} and \eqref{Menon_id_char_primit} are over $(a,n)=1$, since $\chi(a)=0$ for $(a,n)>1$.  Zhao and Cao \cite{ZhaCao} first proved formula \eqref{Menon_id_char_primit} and then deduced identity \eqref{Menon_id_char} by
using the fact that every Dirichlet character is induced by a primitive character. They showed
that the left hand sides of \eqref{Menon_id_char} and \eqref{Menon_id_char_primit} are multiplicative in $n$ and computed their values for prime powers.

Generalizations of \eqref{Menon_id_char} involving $n$-even functions have been deduced by T\'oth \cite{Tot2018}, using a different approach based on direct manipulations of the
corresponding sums, valid for any integer $n\in \N$.  We quote here only the following result, having a simple proof. See \cite[Th.\ 2.7]{Tot2018}. Let $\chi$ be a primitive Dirichlet character (mod $n$),
where $n\in \N$. Let $f$ be an even function (mod $n$) and  let $s\in \Z$. Then
\begin{equation*}
\sum_{k=1}^n f(k-s) \chi(k) = (\mu*f)(n) \chi(s).
\end{equation*}

For the proof, use that for every even function $f$,
\begin{equation*} 
f(k) = f((k,n)) = \sum_{d\mid (k,n)} (\mu*f)(d),
\end{equation*}
and that for a  primitive character $\chi$ (mod $n$), and for any $d\mid n$, $d<n$ and $s\in \Z$,
\begin{equation*}
\sum_{\substack{k=1\\ k\equiv s \text{ (mod $d$)} }}^n \chi(k)=0.
\end{equation*}

We can deduce
\begin{equation*}
\sum_{k=1}^n f(k-s)\chi(k)= \sum_{k=1}^n \chi(k) \sum_{e\mid (k-s,n)} (\mu*f)(e)
\end{equation*}
\begin{equation*}
=\sum_{e\mid n} (\mu*f)(e) \sum_{\substack{k=1\\ k\equiv s \text{\rm (mod $e$)}}}^n  \chi(k) =  (\mu*f)(n)\chi(s).
\end{equation*}

Further generalizations have been given by Chandran, Namboothiri, Tho\-mas \cite{CNT2021}, Chen \cite{Che2019}, Li, Hu, Kim 
\cite{LHK2018,LiHuKim2019Taiw}, T\'oth \cite{Tot2019}.

\subsection{Identities concerning additive characters}

We use the notation $e(x)=e^{2\pi i x}$ and $n=\prod_{p\in \P} p^{\nu_p(n)}$. For every $n\in \N$, $k\in \Z$, 
\begin{equation*}
S(n,k):= \sum_{a=1}^n (a-1,n)\, e(ak/n) = e(k/n) \sum_{d\mid (n,k)} d \varphi(n/d),
\end{equation*}
being an identity related to the discrete Fourier transform (DFT) of the gcd function. See Lemma
\ref{Lemma_gcd_exp}.

Li and  Kim \cite{LiKim2018} proved that for every $n\in \N$, $k\in \Z$ such that
$\nu_p(n)-\nu_p(k)\ne 1$ for every prime $p\mid n$ one has
\begin{equation*} 
S^*(n,k):= \sum_{\substack{a=1\\(a,n)=1}}^n (a-1,n)\, e(ak/n) = e(k/n) \varphi(n) \tau((n,k)),
\end{equation*}
which reduces to Menon's identity \eqref{Menon_id} in the case $n\mid k$. 
Furthermore, Li and Kim \cite[Th.\ 3.5]{LiKim2018} established a formula for $S^*(n,k)$,
valid for every $n\in \N$ and $k\in \Z$. This was generalized and proved using a different method by T\'oth \cite{Tot2020Arab}.

\subsection{Identities with respect to sets of units}

Let $\Z_n^*$ denote the group of units of the ring $\Z_n$. Let $S$ be an arbitrary nonempty subset of the group $\Z_n^*$ and let $\widehat{\Z_n^*}$ 
denote the character group of $\Z_n^*$.
Cai\'uve and Miguel \cite{CaiMig} proved, that for every $n\in \N$,
\begin{equation} \label{Menon_S_Miguel}
M_S(n):= \sum_{a\in S} (a-1,n) = \sum_{\chi \in \widehat{\Z_n^*}} \tau(n/d_{\chi}) \sum_{a\in S} \overline{\chi}(a),      
\end{equation}
where $d_{\chi}$ is the conductor of the character $\chi$, and $\overline{\chi}$ stands for the conjugate of $\chi$. If $S=\Z_n^*$, then
\eqref{Menon_S_Miguel} reduces to Menon's identity \eqref{Menon_id} by using \eqref{sum_values_charact}.

Here is the proof: According to relation \eqref{charct_orthog},
\begin{equation*} 
\delta_x(a):= \frac1{\varphi(n)} \sum_{\chi\in \widehat{Z_n^*}} \overline{\chi}(x) \chi(a) = \begin{cases} 1, & \text{ if $x\equiv a$ (mod $n$)}, \\ 0,  & \text{ otherwise},
\end{cases}  
\end{equation*}
hence
\begin{equation*} 
\delta_S(a):= \sum_{x\in S} \delta_x(a)=  \begin{cases} 1, & \text{ if $a\in S$}, \\ 0,  & \text{ otherwise},
\end{cases}  
\end{equation*}
and deduce that
\begin{equation*} 
M_S(n) = \sum_{a\in \Z_n^*} (a-1,n) \delta_S(a) = \sum_{a\in \Z_n^*} (a-1,n) \sum_{x\in S} \delta_x(a)  
\end{equation*}
\begin{equation*}
=\frac1{\varphi(n)} \sum_{a\in \Z_n^*} (a-1,n) \sum_{x\in S}  \sum_{\chi \in \widehat{\Z_n^*}} \overline{\chi}(x) \chi(a) 
\end{equation*}
\begin{equation*}
=\frac1{\varphi(n)} \sum_{x\in S}  \sum_{\chi \in \widehat{\Z_n^*}} \overline{\chi}(x) \sum_{a\in \Z_n^*} (a-1,n)  \chi(a). 
\end{equation*}

Now application of the result \eqref{Menon_id_char} by Zhao and Cao \cite{ZhaCao} finishes the proof of
\eqref{Menon_S_Miguel}.

We note that Cai\'uve and Miguel \cite{CaiMig} proved \eqref{Menon_S_Miguel} in a more general form, namely for multiple sums
and several subsets of $\Z_n^*$, and investigated the case when these subsets are subgroups of $\Z_n^*$.   

\subsection{Identities derived from group actions by the orbit counting lemma}

Another kind of generalizations of Menon’s identity, based on applying the orbit counting lemma to certain
group actions have been obtained by T\u{a}rn\u{a}\-ucea\-nu \cite{Tar2012JNT}, Li and Kim \cite{LiKim2017Lin}. Namely, they considered certain
subgroups of $\GL_r(\Z_n)$ acting on $\Z_n^r$. For example, taking the subgroup of the upper triangular matrices in $\GL_r(\Z_n)$, 
Li and Kim \cite[Th.\ 2.4]{LiKim2017Lin} proved the following result.

If $n,r\in \N$, then
\begin{equation} \label{Menon_id_LiKim}
\sum_{a_{ij}} \prod_{k=1}^r (d_k/d_{k-1},n) = n^{r(r-1)/2} \varphi(n)^r \tau_{r+1}(n),   
\end{equation}
where the sum is over all $a_{ij}$ ($1\le i,j\le r$) such that $1\le a_{ij}\le n$ ($1\le i,j\le r$) and $(a_{ii},n)=1$ ($1\le i\le r$), 
$d_k$ is the gcd of all $k\times k$ minors of the matrix 
\begin{equation*}
\begin{pmatrix}
a_{11}-1 & a_{12} & \ldots & a_{1r} \\
0 & a_{22}-1 & \ldots & a_{2r} \\
\vdots & \vdots & \ddots & \vdots \\
0 & 0 & \ldots & a_{rr}-1
\end{pmatrix}    
\end{equation*}
for $k\ge 1$, $d_0=1$, by convention $0/0=0$ for $d_k/d_{k-1}$, and $\tau_{r+1}(n) = \sum_{d_1\cdots d_rd_{r+1}=n} 1$ is the Piltz divisor function. 
Note that this corrects the corresponding result obtained by T\u{a}rn\u{a}uceanu \cite{Tar2012JNT}. If $r=1$, then \eqref{Menon_id_LiKim} reduces to
Menon's identity \eqref{Menon_id}.

Some other Menon-type identities have been obtained by Wang and Ji \cite{WJ2019A} and T\u{a}rn\u{a}uceanu \cite{Tar2020Austr}, by using the action of subgroups of $U(\Z_n)$ on the set $\Z_n$,
and the action of the power automorphism group of $G$ on $G$, where $G$ is a finite abelian group, respectively.

\subsection{Identities concerning subsets of the set $\{1,\ldots,n\}$}

For a nonempty subset $A$ of $\{1,2,\ldots, n\}$ let $(A)$
denote the gcd of the elements of $A$. The subset $A$ is said to be
relatively prime if $(A)=1$, i.e., the elements of $A$ are
relatively prime. Let $f(n)$ denote the number of relatively prime
subsets of $\{1,2,\ldots, n\}$. For every $n\in \N$ one has
\begin{equation} \label{f_relprime}
f(n)=\sum_{d=1}^n \mu(d)\left(2^{\lfloor n/d\rfloor}-1\right),
\end{equation}
where $\lfloor x \rfloor$ is the floor function of $x$. 

A similar formula is valid for the number $f_k(n)$ of relatively prime
$k$-subsets (subsets with $k$ elements) of $\{1,2,\ldots, n\}$. Namely, for every $n,k\in \N$ ($k\le n$),
\begin{equation} \label{f_relprime_k}
f_k(n)=\sum_{d=1}^n \mu(d) \binom{\lfloor n/d\rfloor}{k}.
\end{equation}

Furthermore, consider the Euler-type functions $\Phi(n)$ and
$\Phi_k(n)$, representing the number of nonempty subsets $A$ of
$\{1,2,\ldots, n\}$ and $k$-subsets $A$ of $\{1,2,\ldots, n\}$,
respectively, such that $(A)$ and $n$ are relatively prime. Observe that
$\Phi_1(n)=\varphi(n)$ is Euler's function. One has $\Phi(1)=1$, 
\begin{equation*} 
\Phi(n)=\sum_{d\mid n} \mu(d) 2^{n/d} \quad (n\in \N, n>1),
\end{equation*}
and for every $n,k\in \N$ ($k\le n$),
\begin{equation*} 
\Phi_k(n)=\sum_{d\mid n} \mu(d) \binom{n/d}{k}.
\end{equation*}

The functions $f$, $f_k$, $\Phi$ and $\Phi_k$ have been defined and investigated by Nathan\-son
\cite{Nat2007}. Also see T\'oth \cite{Tot2010Integers} and its references.

Define the sum $\overline{M}(n)$ by
\begin{equation*} 
\overline{M}(n):= \sum_{\substack{\emptyset \ne A \subseteq \{1,2,\ldots,n\} \\ ((A),n)=1}} ((A)-1,n),
\end{equation*}
taken over all nonempty subsets of $\{1,2\ldots,n\}$ such that $(A)$ and $n$ are relatively prime,  
where $((A)-1,n)$ denotes the gcd of $(A)-1$ and $n$. 
Also, for $1\le k\le n$ let
\begin{equation*} 
\overline{M}_k(n):= \sum_{\substack{A \subseteq \{1,2,\ldots,n\} \\ \# A =k \\ ((A),n)=1}} ((A)-1,n),
\end{equation*}
the sum being over the $k$-subsets of $\{1,2,\ldots,n\}$ such that $(A)$ and $n$ are relatively prime.
Observe that the sums $\overline{M}(n)$ and $\overline{M}_k(n)$ have $\Phi(n)$, respectively $\Phi_k(n)$ terms.

T\'oth \cite{Tot2021} proved the following identities. For every $n,k\in \N$,
\begin{equation*}
\overline{M}(n)=  \sum_{d\mid n} \varphi(d) \sum_{\substack{\delta \mid n\\ (\delta,d)=1}} \mu(\delta) \sum_{\substack{j=1\\ \delta j \equiv 1 
\textup{ (mod $d$)}}}^{n/\delta} 
f\left(\left\lfloor \frac{n}{j \delta}\right\rfloor\right),
\end{equation*}
\begin{equation} \label{overline_M_k}
\overline{M}_k(n)=  \sum_{d\mid n} \varphi(d) \sum_{\substack{\delta \mid n\\ (\delta,d)=1}} \mu(\delta) \sum_{\substack{j=1\\ \delta j \equiv 1 
\textup{ (mod $d$)}}}^{n/\delta} 
f_k\left(\left\lfloor \frac{n}{j \delta}\right\rfloor\right),
\end{equation}
where the functions $f$ and $f_k$ are given by \eqref{f_relprime} and \eqref{f_relprime_k}, respectively.

Note that if $k=1$, then $f_1(n)=1$ ($n\in \N$), the last sum in \eqref{overline_M_k} is $n/(d\delta)$, and this quickly leads to Menon's identity \eqref{Menon_id}. 
See the proof of Lemma \ref{Lemma_n_d}.

\subsection{Identities in Dedekind domains}

All generalizations and analogs discussed above are in the setting of 
rational integers. Miguel \cite{Mig2014,Mig2016} considered residually finite Dedekind domains, that is, Dedekind domains $\fD$ 
such that for every non-zero ideal $\fn$ of $\fD$, the residue class ring ${\fD}/{\fn}$ is finite. 
Let $N(n)=|{\fD}/{\fn}|$ denote the norm of the ideal $\fn$. Let $\varphi_{\fD}$ denote the generalized  
Euler function defined as 
\begin{equation*}
\varphi_{\fD}({\fn}) = \begin{cases} 1, & \text{ if $\fn={\fD}$},\\
|U({\fD}/\fn)|, & \text{ otherwise}.
\end{cases}    
\end{equation*}

Furthermore, let $\tau_{\fD}(n)= \sum_{\fd \mid \fn} 1$ be the corresponding divisor function. 
Note that $\varphi_{\fD}$
and $\tau_{\fD}$ are finite valued functions.
Miguel \cite[Th.\ 1.1]{Mig2014} proved by applying Burnside's lemma that for every non-zero ideal 
$\fn$ in ${\fD}$,
\begin{equation} \label{Menon_res_Dedekind}
\sum_{a\in \U({\fD}/{\fn})} N(\langle a-1 \rangle + \fn)  = \varphi_{\fD}(n) \tau_{\fD}(\fn),
\end{equation}
which reduces to Menon's identity \eqref{Menon_id} in the case ${\cal D}=\Z$. 

Miguel \cite[Th.\ 1.1]{Mig2016} also deduced, again by Burnside's lemma, the generalization of Sury's identity 
\eqref{Sury_id}. A further generalization was given by Li and Kim \cite{LiKim2017Ded}.

Identity \eqref{Menon_res_Dedekind} has been generalized with respect to regular systems of divisors and $k$-reduced residue systems
by Wang, Zhang and Ji \cite{WZJ2019B}. The Sita Ramaiah identity, which is the case $k=2$ of \eqref{Menon_gen_k} has been generalized
in the setting of residually finite Dedekind domains with respect to regular systems of divisors and $k$-reduced residue systems 
by Ji and Wang \cite{JW2020}. Also see Wang, Hu, and Ji \cite{WHJ2019} for some different results deduced by the orbit counting lemma,
and Chen and Zheng \cite{CheZhe2021} for some further generalizations.

\subsection{Identities in the ring of algebraic integers}

Let $K$ be an algebraic number field with the ring of integers denoted by ${\cal O}_K$. Then ${\cal O}_K$ is a residually finite Dedekind domain.
Let $\fn$ be a non-zero ideal of ${\cal O}_K$ and let $\chi$ be a character (mod $\fn$) with conductor $\fd$. 
Wang and Ji \cite[Cor.\ 1.3]{WJ2020} proved, as a corollary of more general results, that 
\begin{equation} \label{Menon_number_field}
 \sum_{a\in U({\cal O}_K/{\fn})} N(\langle a-1 \rangle +\fn ) \chi(a) = \varphi(\fn) \tau(\fn/\fd),
\end{equation}
where $N$, $\varphi$ and $\tau$ are the corresponding norm, Euler and divisor functions, respectively. If $\chi$ is the principal character, then 
\eqref{Menon_number_field} reduces to identity \eqref{Menon_res_Dedekind} in the case of algebraic integers, and if $K=\Q$, then \eqref{Menon_number_field} recovers \eqref{Menon_id_char}.

Identities \eqref{Menon_gen_f} and \eqref{Menon_gen_k} by T\'oth have been generalized by Sarkar \cite{Sar2021arXiv} to the ring of algebraic integers ${\cal O}_K$ and 
involving Dirichlet characters. The case $k=2$, has been earlier investigated by Wang and Ji \cite{WJ2020Raman},
Chattopadhyay and Sarkar \cite{ChaSar2020arXiv}.

\subsection{Final remarks}

Extensions of Menon-type identities, in particular of the general identity by T\'oth \cite{Tot2019}, to the polynomial ring over the a finite field $\F_q$ 
have been obtained Chen and Qi \cite{CQ2021}.

Some further papers related to Menon-type identities, not mentioned above in this survey are \cite{CHL2018arXiv, Che2019kinai, HauWan1997, Li2021Miskolc, MS2022, 
Nat2021, Venkataraman1974, Ven1984, Ven1997}.

\end{document}